\documentclass[12pt]{amsart}
\usepackage{amsfonts}
\usepackage{bbm}
\usepackage{mathrsfs}
\usepackage{amssymb,amsmath,amsfonts,amsthm, color}

\usepackage{txfonts}

\usepackage[
colorlinks]{hyperref}

\numberwithin{equation}{section}

\begin{document}

\newtheorem{theorem}{Theorem}[section]
\newtheorem{prop}[theorem]{Proposition}
\newtheorem{lemma}[theorem]{Lemma}
\newtheorem{definition}[theorem]{Definition}
\newtheorem{corollary}[theorem]{Corollary}
\newtheorem{example}[theorem]{Example}
\newtheorem{remark}[theorem]{Remark}

\def \RN  {\mathbb{R}^n}
\def \R  {\mathbb{R}}

\newcommand{\ra}{\rightarrow}
\newcommand{\ccc}{{\mathcal C}}
\newcommand{\one}{1\hspace{-4.5pt}1}
\newcommand{\omegareverse}{\frac{\displaystyle 1}{\displaystyle
|\Omega|}}


\newcommand{\XX}{\mathcal {X}}
\newcommand{\GG}{\mathop \mathcal{G} \limits^{    \circ}}



\title[Frames decompositions on Hardy spaces]
 {Frame decomposition and radial maximal semigroup
  characterization of Hardy spaces associated to\\
   operators }


\author{  Xuan Thinh Duong }

\address{Department of Mathematics, Macquarie University, NSW, 2109, Australia}

\email{xuan.duong@mq.edu.au}

\author{  Ji Li }

\address{Department of Mathematics, Macquarie University, NSW, 2109, Australia}

\email{ji.li@mq.edu.au}

\author{ Liang Song }

\address{Department of Mathematics, Sun Yat-sen University, China}

\email{songl@mail.sysu.edu.cn}

\author{  Lixin Yan }

\address{Department of Mathematics, Sun Yat-sen University, China}

\email{mcsylx@mail.sysu.edu.cn}


\begin{abstract} Let $L$ be the generator of an analytic semigroup whose kernels satisfy Gaussian upper bounds and H\"older's continuity.
Also assume that $L$ has a bounded holomorphic functional calculus on $L^2(\mathbb{R}^n)$.
In this paper, we construct a frame decomposition for the functions belonging to the Hardy space  $H_{L}^{1}(\mathbb{R}^n)$
associated to $L$, and for functions in the Lebesgue spaces $L^p$, $1<p<\infty$. We then show that
the corresponding $H_{L}^{1}(\RN)$-norm (resp. $L^p(\RN)$-norm) of a function $f$ in terms of the frame coefficients is equivalent to the  $H_{L}^{1}(\RN)$-norm (resp. $L^p(\RN)$-norm) of $f$.
As an application of  the frame decomposition, we establish the radial maximal
semigroup characterization of the Hardy space  $H_{L}^{1}(\mathbb{R}^n)$ under the extra condition
of Gaussian upper bounds on the gradient of the heat kernels of $L$.
\end{abstract}

\maketitle


\tableofcontents

\allowdisplaybreaks

\section{Introduction and statement of results}
\setcounter{equation}{0}

Wavelet analysis has played an important role in many different branches of science and technology since
it provides a simple and efficient way, in addition to Fourier series and integrals,  to analyse functions and distributions.
The wavelet series decompositions are effective
expansion by unconditional
 bases in the standard Lebesgue spaces $L^p(\RN)$, $1<p<\infty$, as well as many other spaces such as
  Hardy spaces, BMO spaces, Besov spaces which arise in the theory of harmonic analysis. A function $f$
   (may be tempered distributions in some cases) in these various spaces can thus be written in the form
\begin{align}\label{wavelet expansion}
 f(x) =\sum_{\lambda \in \Lambda} \langle f,\psi_\lambda \rangle \psi_\lambda(x),
\end{align}
and the series converges unconditionally to $f(x)$ in the relevant norm. Here
 $\Lambda=\{ \lambda=2^{-j}k+\epsilon2^{-j-1}: j\in\mathbb{Z}, k\in\mathbb{Z}^n,
  \epsilon\in E \}$, where $E=\{0,1\}^n$ excluding $(0,0,\ldots,0)$. The family $\psi_\lambda,
   \lambda\in \Lambda$ is a wavelet basis arising from an $r$-regular multiresolution
   approximation of $L^2(\RN)$. Moreover,
the norms of elements in these various spaces can be equivalently characterized by the corresponding
 norms via coefficients of the expansion in \eqref{wavelet expansion}. To be more precise,
 taking $L^p(\RN)$, $1<p<\infty$ and Hardy space $H^1(\RN)$ for example, we have
$$ \|f\|_{L^p(\RN)} \approx \Big\| \Big\{\sum_{\lambda\in\Lambda}
 |\langle f,\psi_\lambda\rangle|^2 |Q_\lambda|^{-1}\chi_\lambda(\cdot) \Big\}^{1\over2} \Big\|_p $$
and
$$ \|f\|_{H^1(\RN)} \approx \Big\| \Big\{\sum_{\lambda\in\Lambda}
|\langle f,\psi_\lambda\rangle|^2 |Q_\lambda|^{-1}\chi_\lambda(\cdot)
\Big\}^{1\over2} \Big\|_1, $$
where $Q_\lambda$ is the dyadic cube defined by $2^jx-k\in [0,1)^n$
 and $\chi_\lambda(x)$ is the characteristic function of $Q_\lambda$.
 For more details about the wavelet theory, we refer to \cite{Ch, Da, FJW, Me, Meyer}.
 We note that wavelet theory has also been developed in many other settings including that wavelet bases being
 replaced by frames which offer the same service  in many applications. The success of wavelet theory
 lies in the fact that it has had applications in widely differing areas of science, see for
 example \cite{BL, BCMS, Coi, CM,  Da2,DH}  and the references therein.

The classical theory of Hardy spaces on $\RN $ has been a great success and central to the
estimates of singular integrals \cite{St}. Since a number of characterizations of the classical Hardy space can be given via
various estimates of the Laplace operator, one can say that the classical Hardy space
is associated to the Laplace operator.
We note that the Laplace operator has its heat kernel $p_t(x,y)$ given explicitly by the Gaussian kernel,
hence all the heat kernel regularity such as the time derivatives (to all order) and spacial derivatives can be
computed explicitly. The Laplace operator also possesses the conservation property and is non-negative self-adjoint,
therefore it has a bounded functional calculus on $L^2(\RN)$ for bounded measurable functions on $[0, +\infty)$.

In the last decade, a theory of function spaces,
and in particular Hardy spaces, associated to an operator $L$ was developed and studied extensively. This theory has arisen
from the need of studying singular integrals with non-smooth kernels which do not belong to the so-called class
of Calder\'on-Zygmund operators. In this theory, the assumptions on the heat kernel of $L$ and the functional calculus of $L$
play a key role. The weaker these assumptions are, the more operators $L$ are included in the theory but less features and
characterizations of the spaces can be obtained. We now list a number of articles closely related
to the development of this topic but our list is by no means exhaustive.

(i) In \cite{AuDM}, P. Auscher, X.T. Duong and A. M$^{\rm c}$Intosh
introduced  the Hardy space $H_L^1(\mathbb{R}^n)$ associated to an
operator $L$, and obtained a molecular decomposition, assuming that
$L$ has a bounded holomorphic functional calculus on
$L^2(\mathbb{R}^n)$ and the kernel of the heat semigroup $e^{-tL}$
has a pointwise Poisson upper bound. Under the same assumptions on $L$, X.T. Duong and L.X. Yan
introduced the space $BMO_L(\mathbb{R}^n)$ adapted to $L$ and established the
duality of $H_L^1(\mathbb{R}^n)$ and $BMO_{L^*}(\mathbb{R}^n)$ in
\cite{DY1}, \cite{DY2}, where $L^*$ denotes the adjoint operator of
$L$ in $L^2(\mathbb{R}^n)$.
Later,   the Hardy spaces $H_L^p(\mathbb{R}^n)$
for  all $0<p<1$ were  established in \cite{Yan}.

(ii)  P. Auscher,  A. M$^{\rm c}$Intosh  and E. Russ \cite{AMR} established the Hardy spaces $H^p_L,
p\geq 1,$ associated to the Hodge Laplacian on a Riemann manifold with
doubling measure. S. Hofmann and S. Mayboroda \cite{HM} defined the Hardy spaces $H^p_L,
p\geq 1$, associated to a second order divergence form elliptic
operator on $\mathbb{R}^n$ with complex coefficients.
In these
settings, pointwise heat kernel bounds may fail. By making  use of the
notion of ``$L$-cancellation" of molecules, they  studied  the Hardy
space $H^1_L$ including a molecular decomposition, a square function
characterization, its dual space and others properties.

(iii) Later, in \cite{HLMMY}, S. Hofmann {\it et al} developed the theory of $H^1$ and $BMO$ spaces
adapted to a non-negative, self-adjoint operator $L$ whose heat kernel
satisfies the weak Davies-Gaffney bounds, in the setting of a
space of homogeneous type $X$.  For the Hardy space
$H^1_L(X)$, they also obtained an atomic decomposition.  X.T. Duong and J. Li \cite{DL}  extended this line to develop
Hardy spaces $H_L^p(X)$ for $0<p\leq 1$, including a molecular
decomposition, a square function characterization, duality of Hardy
and Lipschitz spaces, and a Marcinkiewicz type interpolation
theorem, where the operator $L$ needs not be a non-negative self-adjoint operator. R.J. Jiang and D.C. Yang \cite{JY2} also extended this line to
Orlicz--Hardy spaces.

(iv) X.T. Duong, J. Li and L.X. Yan \cite{DLY} established a discrete characterization
   of weighted
 Hardy spaces    $H_{L, S, w}^{p}(X)$
 associated to $L$ in terms of
  the area function characterization, where $L$ is a second order non-negative self-adjoint
 operator  on $L^2(X)$ satisfying the Moser-type condition, and the semigroup $e^{-tL}$ generated by $L$
 satisfies Gassian upper bounds.

 (v) In \cite{SY1}, L. Song and L.X. Yan  used
 a modification of technique due to A. Calder\'on \cite{C} to give an atomic decomposition for the Hardy spaces
 $ H^p_{L,max}(\mathbb{R}^n)$  in terms of the nontangential maximal functions associated with the heat semigroup
 of $L$, where $L$ is a second order non-negative self-adjoint
 operator  on $L^2(X)$ and its heat semigroup satisfying   Gaussian estimates
 on $L^2(\mathbb{R}^n)$.  This leads   eventually to characterizations of Hardy spaces associated to $L$,  via
  atomic decomposition or the nontangential  maximal functions.
    In term of the  radial maximal function
   characterization    of Hardy spaces,
    D.C. Yang and S.B. Yang   \cite{YY}  obtained it    under the additional assumption that the kernel of the heat semigroup satisfies H\"older's continuity.
   In \cite{SY2}, L. Song and L.X. Yan got rid of the additional assumption of \cite{YY} and proved
   the radial maximal function characterization of Hardy spaces associated to $L$.

 (vi)  Recently,  G. Kerkyacharian and P. Petrushev  \cite{KP} introduced a nice frame decomposition (associated to $L$)
 for the Schwarz functions and distributions  and established the Besov and Triebel-Lizorkin spaces associated with
 $L$ in the framework of Dirichlet spaces with a doubling measure $\mu$ satisfying also the reverse doubling condition
 and the non-collapsing condition, where the operator $L$ is self-adjoint and satisfies the small time Gaussian upper
 bound, the H\"older continuity as well as the preservation property  (Markov property) , i.e., $e^{-tL}1=1$. Later, S. Dekel, G. Kerkyacharian, G. Kyriazis, P. Petrushev \cite{DKKP}
obtained a compactly supported frames for spaces of distributions  associated with non-negative self-adjoint operators satisfying the preservation property (Markov property)  on a more general setting: spaces of homogeneous type. The development of such frames is important in a situation where no additional structures such as translation invariance or a dilation operator are present.


So far, the main characterisations of Hardy spaces associated to operators are obtained via area integral estimates, atomic
or molecular decompositions, and  maximal functions. We observe that the frame structure is absent so far  for  these Hardy spaces when $L$ is either
non-selfadjoint or $e^{-tL}$ does not satisfy the preservation property  (Markov property) or both.

The aim of this paper is to
obtain a frame decomposition for the functions belonging to the Hardy space  $H_{L}^{1}(\RN)$
associated to $L$, as well as for the functions in the Lebesgue spaces $L^p$, $1<p<\infty$, where the corresponding $L^1$
 norm  (resp. $L^p$-norm) of a function $f$ in terms of the frame coefficients is equivalent to the  $H_{L}^{1}(\RN)$-norm (resp. $L^p$-norm) of $f$ (see Theorems
\ref{main1} and \ref{theorem  wavelet expansion of H1} below).
As an application, we establish the radial maximal semigroup characterization of the Hardy space  $H_{L}^{1}(\mathbb{R}^n)$ by using  the frame decomposition
(see Theorem \ref{th5.1} below).

We now state our assumptions and main results.

Let  $L$ be a linear operator of type $\omega$
 ($\omega<\pi/2$), which is one-one with dense range on $L^2({\mathbb R}^{n})$, hence $L$ generates a holomorphic semigroup
$e^{-zL}$, $0\leq |{\rm Arg}(z)|<\pi/2-\omega$ (for more details we refer the readers to the beginning of Section 2).  The following shall
be assumed throughout the paper  unless otherwise specified:

\smallskip

\noindent
${\bf(H 1)}$ The operator $L$ has a bounded
$H_{\infty}$-calculus on  $L^2({\mathbb R}^{n})$. That is, there exists
$c_{\nu,2}>0$ such that  $b(L)\in {\mathcal L}(L^2, L^2)$,  and for
$b\in H_{\infty}(S^0_{\nu}):$
\begin{eqnarray*}
\|b(L)f\|_2\leq c_{\nu,2}\|b\|_{{\infty}}\|g\|_2
\end{eqnarray*}
\noindent
for any  $f\in L^2({\mathbb R}^{n})$. See Section 2 for the definition of $H_{\infty}(S^0_{\nu})$ and for more details of this assumption.

\smallskip

\noindent ${\bf(H 2)}$ The analytic semigroup $\{e^{-zL}\}$, $|Arg(z)|<\pi/2-\omega$, is represented by the  kernel  $p_z(x,y)$
which  satisfies the following  Gaussian upper bound
$$
|p_{z}(x,y)| \leq C_\theta \frac{1}{|z|^{n/2} } \exp\Big(-{|x-y|^2\over
c\,|z|}\Big)
\leqno{(GE)}
$$
for all $x,y\in {\mathbb R}^{n},$ $|Arg(z)|<\pi/2-\theta$ for $\theta>\omega$.



\smallskip

\noindent ${\bf(H 3)}$ The analytic semigroup $\{e^{-zL}\}$, $|Arg(z)|<\pi/2-\omega$, is represented by the  kernel  $p_z(x,y)$
which  satisfies the following  regularity
$$
|p_{z}(x,y)-p_{z}(x',y)|+|p_{z}(y,x)-p_{z}(y,x')| \leq C_\theta \Big( {|x-x'|\over |z|^{1/2}+|x-y| }\Big)^\gamma   \frac{1}{|z|^{n/2} } \exp\Big(-{|x-y|^2\over
c\,|z|}\Big)
$$
for some $\gamma\in(0,1]$ and for all $x, x', y\in {\mathbb R}^{n}$ with $2|x-x'|\leq |z|^{1/2}+|x-y| $,  $|Arg(z)|<\pi/2-\theta$ for $\theta>\omega$.


Suppose $\zeta\in H(S^0_{\nu})$ with $\zeta\not\equiv0$ and
$$ |\zeta(z)|\leq C{|z|^\alpha\over
1+|z|^{\beta}}
$$
where $z\in S^0_{\nu},\ \  \alpha>0, \  \beta>\alpha+n+\gamma+3$ in which
$n$ is the dimension and $\gamma$ is the constant in the assumption ${\bf(H 3)}$. Put
\begin{align}\label{varphi}
q(z)=z^2 \zeta^2(z), \  q_t(z)=q(tz).
\end{align}

Let $\delta$ be a constant satisfying $1<\delta<2$. For each $j$, let $I_j$ denote the net of $\delta$-dyadic cubes with side-length $\delta^{-j-M}$ with a  large fixed positive integer $M$, where one such cube in the net has the origin as the lower left  vertex. And let $\tau$ be the index in $I_j$ and $Q_{\tau}^j$ denote the cube belong to $I_j$, and $y_{Q_\tau^j}$ denote the centre of the cube $Q_{\tau}^j$.

Denote $q_j(x,y)$ the kernel of the operator $q_{\delta^{-2j}}(L)$ with $q_t$ defined as in \eqref{varphi} above  \big(where $\delta$ is to be determined later\big).
Also, set
\begin{align}
\psi_{j,\tau}(x):=\sqrt{\ln\delta}|Q_{\tau}^j|^{1/2}q_j(x,y_{Q_{\tau}^j})\ \ {\rm and} \ \     \psi^*_{j,\tau}(x):=\sqrt{\ln\delta}|Q_{\tau}^j|^{1/2}\overline{q_j(y_{Q_{\tau}^j},x)}
\end{align}
for any $y_{Q_{\tau}^j}\in Q_{\tau}^j$. We point out that $q_j(x,y)$ is continuous in both $x$ and $y$ (see Proposition \ref{propHeat2} below), hence the functions $\psi_{j,\tau}(x)$ and $\psi^*_{j,\tau}(x)$ are well-defined for any $y_{Q_{\tau}^j}\in Q_{\tau}^j$.

Next for any $f\in L^2(\RN)$, we define the auxiliary operator
\begin{align}\label{Tdelta}
T_ \delta(f)(x):=\ln\delta\sum\limits_j\sum\limits_{\tau\in I_j}|Q_{\tau}^j| q_j(x,y_{Q_{\tau}^j})q_{\delta^{-2j}}(L)f(y_{Q_{\tau}^j}),
\end{align}
where $y_{Q_{\tau}^j}$ is any point in the cube $Q_{\tau}^j$ and $$ q_{\delta^{-2j}}(L)(f)(y_{Q_{\tau}^j}) = \int_{\RN} q_j(y_{Q_{\tau}^j},y)f(y)dy.$$
To see that $T_\delta$ is well-defined and bounded on $L^2(\RN)$, we refer to Section 3 below.



The first  main result in this paper is  the following frame decomposition of the functions in $L^p(\mathbb{R}^n)$, $1<p<\infty$.

\begin{theorem}\label{main1}
Assume that $L$ satisfies ${\bf(H 1)}$, ${\bf(H 2)}$ and ${\bf(H 3)}$. Let $1<p<\infty$ and $f\in L^p(\mathbb{R}^n)$.
Then we
have the following frame decomposition of $f$
\begin{eqnarray}\label{wavelet expansion of Lp}
 f=\sum\limits_j\sum\limits_{\tau\in I_j}\langle T_\delta^{-1}f,
 \psi^*_{j,\tau}\rangle  \psi_{j,\tau} \hskip 2cm  in \ \ L^p(\mathbb{R}^n).
\end{eqnarray}
Furthermore, there exist positive constants $C_1$ and $C_2$, such
that
\begin{eqnarray}\label{wavelet norm of Lp}
C_1\big\|f\big\|_p \leq \Big\|\Big( \sum\limits_j
\sum\limits_{\tau\in I_j}|\langle T_\delta^{-1}f, \psi^*_{j,\tau}\rangle|^2
\frac{1}{|Q_{\tau}^j|}\chi_{Q_{\tau}^j}\Big)^{1\over 2}\Big\|_p\leq
C_2\big\|f\big\|_p.
\end{eqnarray}
\end{theorem}

Next we recall the definition of Hardy space associated to $L$.

\begin{definition}[\cite{DY2,HLMMY}]\label{def-Hardy space via S function}
Assume that $L$ satisfies ${\bf(H 1)}$ and ${\bf(H 2)}$.
The Hardy space $H^1_{L,  S_L}({\mathbb R}^n) $ is defined as the completion of
$\{ f\in L^2({\mathbb R}^n):\, \|S_{L} f\|_{L^1({\mathbb R}^n)}<\infty \}
$, with  norm
$
 \|f\|_{H^1_{L, S_L}({\mathbb R}^n)}:=\|S_{L} f\|_{L^1({\mathbb R}^n)},
$ where
$$S_L(f)(x)=\bigg(\int_0^{\infty}\int_{|x-y|<t}\big|q_{t^2}(L)(f)(y)\big|^2{dydt\over
t^{n+1}}\bigg)^{1\over 2}.$$

\end{definition}

The second  main result in this paper is  the following frame decomposition of the functions in $H_L^1(\mathbb{R}^n)$.

\begin{theorem}\label{theorem  wavelet expansion of H1}
Assume that $L$ satisfies ${\bf(H 1)}$, ${\bf(H 2)}$ and ${\bf(H 3)}$. For every $f\in H^1_{L,  S_L}({\mathbb R}^n) $,
 we have the following frame decomposition
\begin{eqnarray}\label{wavelet expansion of H1}
 f=\sum\limits_j\sum_{\tau\in I_j}\langle T_\delta^{-1}f,
 \psi^*_{j,\tau}\rangle  \psi_{j,\tau} \hskip 2cm  in \ \ H_L^1(\mathbb{R}^n).
\end{eqnarray}
Furthermore, there exist positive constants $C_1$ and $C_2$, such
that
\begin{eqnarray}\label{wavelet norm of H1}
\hskip1cm C_1\big\|f\big\|_{H_L^1(\mathbb{R}^n)} \leq \Big\|\Big(
\sum\limits_j \sum_{\tau\in I_j}|\langle T_\delta^{-1}f,
\psi^*_{j,\tau}\rangle|^2
\frac{1}{|Q_{\tau}^j|}\chi_{Q_{\tau}^j}\Big)^{1\over
2}\Big\|_{L^1(\mathbb{R}^n)}\leq
C_2\big\|f\big\|_{H_L^1(\mathbb{R}^n)}.
\end{eqnarray}
\end{theorem}

\begin{remark} Operators which satisfy the assumptions ${\bf(H 1)}$, ${\bf(H 2)}$ and ${\bf(H 3)}$ include the following:

(i) Laplace operator on the Euclidean spaces $\mathbb{R}^n$.

(ii) Second order elliptic divergence form operators with bounded, real coefficients
 on $\mathbb{R}^n$. See \cite{AT}.

(iii) The Schr\"odinger operators $-\Delta + V$ on $\mathbb{R}^n$ where the potentials $V$ belong to a suitable H\"older class. See, for example, \cite{DZ}.
\end{remark}

\begin{remark}
Note that in   Theorems \ref{main1} and \ref{theorem  wavelet expansion of H1} above:

(i)  The operators $q_{\delta^{-2j}}(L)$ can be replaced by
$  \big(\delta^{-2j}L\big)^ke^{-\delta^{-2j}L} $ for any $k\in \mathbb{N}$, which can be observed from the strategy of our proofs.
Hence the function $\psi_{j,\tau}(x)=\sqrt{\ln\delta}|Q_{\tau}^j|^{1/2}\big(\delta^{-2j}L\big)^ke^{-\delta^{-2j}L}(x,y_{Q_{\tau}^j})$, where $\big(\delta^{-2j}L\big)^ke^{-\delta^{-2j}L}(x,y)$ is the kernel of the operator $\big(\delta^{-2j}L\big)^ke^{-\delta^{-2j}L} $.

(ii) If $L$ is non-negative self-adjoint, then $\psi^*_{j,\tau}$ equals $\psi_{j,\tau}$.
\end{remark}

\smallskip
Our strategy of proof is the following:

\smallskip
Step 1. We first develop the following two key technical results:

(i) Using the holomorphic functional calculi of operators, we obtain the almost orthogonality
 estimates for the operators $\{q(tL)\}_{t>0}$, where  the function $q(\lambda)$ is defined as in \eqref{varphi}. See more details in  Proposition \ref{s-prop3.3}.

(ii) We introduce four different versions of the discrete Littlewood--Paley $g$-functions associated to the operator $L$,
and by applying the almost orthogonality estimates for the operators $\{q_t(L)\}_{t>0}$ above, we prove that these
 $g$-functions are bounded on $L^p(\mathbb{R}^n)$, $1<p<\infty$, and on $H_L^1(\mathbb{R}^n)$. See Lemmas \ref{le2.1} and \ref{lemma H1 to L1 molecule}.


\smallskip
Step 2. Let the operator $R_\delta= I-T_\delta$, where $I$ is the identity operator and $T_\delta$ is defined as in \eqref{Tdelta}. Using the two key technical results above, we can
obtain the  operator norms of $R_\delta$ on $L^p(\mathbb{R}^n)$, $1<p<\infty$, and on $H_L^1(\mathbb{R}^n)$. By choosing $\delta>1$ and close to 1,  we get that the operator norms of $R_{\delta}$ are strictly less than 1. This shows that $T_\delta$ is invertible and the inverse operator $T_\delta^{-1}$ is bounded on $L^p(\mathbb{R}^n)$, $1<p<\infty$, and on $H_L^1(\mathbb{R}^n)$. See more details in Theorems \ref{th2.2} and \ref{theorem error term of H1}.

\smallskip

In Section 5,  as an application of the frame decomposition, we apply our main result,
Theorem \ref{theorem  wavelet expansion of H1} and to obtain the radial maximal function characterization of the Hardy space $H_L^1(\mathbb{R}^n)$ under the extra assumption of Gaussian upper bounds on the
gradient of the heat kernels of $L$.

The paper is organised as follows. In Section 2 we prove the almost orthogonality estimates for the operators $\{q_t(L)\}_{t>0}$, where  the function $q_t(z)$ is defined as in \eqref{varphi}. In Section 3 we prove Theorem \ref{main1} by showing that $T_ \delta^{-1}$ exists and is bounded on $L^p(\mathbb{R}^n)$. In Section 4 we prove Theorem \ref{theorem  wavelet expansion of H1} by showing that $T_ \delta^{-1}$ exists and is bounded on $H_L^1(\mathbb{R}^n)$.
The last section is devoted to the proof of the radial maximal function characterization of the Hardy space $H_L^1(\mathbb{R}^n)$.

\vskip.2cm
\section{Notation and preliminaries}
\setcounter{equation}{0}

We first recall some preliminaries on holomorphic functional calculi of operators. See \cite{Mc}.

Let $0\leq \omega<\nu<\pi$. We define the closed sector in the
complex plane ${\mathbb C}$
$$ S_{\omega}=\{z\in {\mathbb C}: |{\rm
arg}z|\leq\omega\}\cup\{0\}
$$
and denote the interior of $S_{\omega}$ by $S_{\omega}^0$. We define
the following subspaces of the space $H(S_{\nu}^0)$ of all
holomorphic functions on $S_{\nu}^0$:
$$
H_{\infty}(S^0_{\nu}) =\{b\in H(S_{\nu}^0):\
||b||_{{\infty}}<\infty\},
$$
\noindent where $||b||_{\infty}={\rm sup}\{|b(z)|: z\in
S^0_{\nu}\}$,  and
$$ \Psi(S^0_{\nu})=\big\{\psi\in H(S^0_{\nu}):
\exists\  \!  s>0,  \ |\psi(z)|\leq c {|z|^s\over {1+|z|^{2s}} }\big\}.
$$
 \noindent Let $0\leq\omega<\pi$. A closed operator $L$ in
$L^2(\RN)$ is said to be of type $\omega$ if  $\sigma(L)\subset
S_{\omega}$, and for each $\nu>\omega,$ there exists a constant
$c_{\nu}$ such that $ \|(L-\lambda {\mathcal I})^{-1}\| \leq
c_{\nu}|\lambda|^{-1},\  \lambda\not\in S_{\nu}. $ \noindent If $L$
is of type $\omega$ and $\psi\in \Psi(S^0_{\nu})$, we define
$\psi(L)\in {\mathcal L}(L^2, L^2)$ by
\begin{equation}\label{functional calculus}
\psi(L)=\frac{1}{2\pi i}\int_{\Gamma}(L-\lambda {\mathcal
I})^{-1}\psi(\lambda)d\lambda,
\end{equation}
\noindent
where $\Gamma$ is the contour $\{\xi=re^{\pm i\theta}: r\geq 0\}$
parametrized clockwise around $S_{\omega}$, and $\omega<\theta<\nu$.
Clearly, this integral is absolutely
convergent in ${\mathcal L}(L^2, L^2)$, and it is straightforward to show that,
using Cauchy's theorem, the definition is independent of the choice of
$\theta\in (\omega, \nu).$  If, in addition, $L$ is one-one and has dense range
and if
$b\in H_{\infty}(S^0_{\nu})$, then $b(L)$ can be defined by
\begin{eqnarray*}
b(L)=[\psi(L)]^{-1}(b\psi)(L),
\end{eqnarray*}
\noindent
where $\psi(z) =z(1+z)^{-2}$. It can be shown that $b(L)$ is a
well-defined linear operator in $L^2(\mathbb{R}^n)$.

We point out that if a closed operator $L$ in
$L^2(\RN)$ is of type $\omega$, then $L^*$ is also of type $\omega$, see \cite[Page 20]{ADM}.

We say that
$L$ has a bounded $H_{\infty}$ calculus in $L^2 $ if there exists
$c_{\nu,2}>0$ such that  $b(L)\in {\mathcal L}(L^2, L^2)$,  and for
$b\in H_{\infty}(S^0_{\nu})$,
\begin{eqnarray*}
||b(L)|| \leq c_{\nu,2}||b||_{{\infty}}.
\end{eqnarray*}
In \cite{Mc} it was proved that $L$ has a bounded
$H_{\infty}$-calculus in  $L^2(\mathbb{R}^n)$ if and only if  for any
non-zero function $\psi\in \Psi(S^0_{\nu})$, $L$  satisfies the
square function estimate    and its reverse
\begin{eqnarray}\label{L2 esti of square function}
c_1\|g\|_2\leq \Big( \int_0^{\infty}\|\psi_t(L)g\|_2^2 {dt\over t
}\Big)^{1/2}\leq c_2\|g\|_2
\end{eqnarray}
\noindent
for some $0<c_1\leq c_2<\infty$,
where $\psi_t(\xi)=\psi(t\xi)$. Note that different choices of
$\nu>\omega$ and $\psi\in \Psi(S^0_{\nu})$ lead to equivalent
quadratic norms of $g$.

Note that by Corollary E in \cite[Page 22]{ADM}, if $L$ satisfies \eqref{L2 esti of square function},
then $L^*$ also satisfies \eqref{L2 esti of square function}.

As noted in \cite{Mc},  non-negative self-adjoint operators satisfy the
quadratic estimate (\ref{L2 esti of square function}). So do normal
operators with spectra in a sector, and maximal accretive operators.
For further study of holomorphic functional calculi on Banach spaces, see \cite{Mc} and \cite{CDMY}.

\medskip

\begin{prop}\label{prop derivative}
 Suppose $\psi\in H(S^0_{\nu})$ with
two parameters $\alpha>0, \beta>\alpha$ such that
\begin{eqnarray}\label{q1}
|\psi(z)|\leq C{|z|^\alpha\over
1+|z|^{\beta}}.
\end{eqnarray}
 Then for each fixed $k\in\mathbb{N}$, we have $ \psi^{(k)}   $ is holomorphic in $ S^0_{\omega}$   for some $\omega$
with $\omega<\nu$, and
\begin{eqnarray*}
| \psi^{(k)}(z) |
\leq   C {1\over |z|^k}{|z|^\alpha\over {1+|z|^{\beta}}}.
\end{eqnarray*}
\end{prop}
\begin{proof}
For each $\psi\in \Psi(S^0_{\nu})$, we have $|\psi(z)|\leq C {|z|^\alpha\over {1+|z|^{\beta}}}$. Now fix $\epsilon>0$ such that $\sin (k\epsilon) <{1\over 100}$, we consider the sector
$S^0_{\nu-\epsilon}$. For every $z\in S^0_{\nu-\epsilon}$ and $z\not=0$, we consider the ball  $B(z,r)$, centered at $z$, with radius $r=|z|\sin\epsilon$,
such that $B(z,r)$ is contained in $S^0_{\nu}$. Then by Cauchy's formula we obtain that
\begin{eqnarray*}
 \psi^{(1)}(z) = {1\over 2\pi i} \oint_{\partial B} {\psi(\lambda)\over (\lambda-z)^2} d\lambda,
\end{eqnarray*}
where $\partial B$ is the boundary of the ball $B(z,r)$.

Then we have
\begin{eqnarray*}
| \psi^{(1)}(z) | &\leq&  {1\over 2\pi } \oint_{\partial B} {|\psi(\lambda)|\over |(\lambda-z)^2|} |d\lambda|\\
&\leq&   {1\over 2\pi r^2 } \oint_{\partial B} |\psi(\lambda)| |d\lambda|.
\end{eqnarray*}
Next, note that $|\lambda|\leq |\lambda-z|+|z| \leq {101\over 100} |z| $, and $|\lambda|\geq |z|-|\lambda-z| \geq {99\over 100} |z|$
\begin{eqnarray*}
|\psi(\lambda)|\leq C {|\lambda|^\alpha\over {1+|\lambda|^{\beta}}} \leq C \Big({101\over 100}\Big)^\alpha\Big({100\over 99}\Big)^{\beta}  {|z|^\alpha\over {1+|z|^{\beta}}}.
\end{eqnarray*}
Thus,
\begin{eqnarray*}
| \psi^{(1)}(z) |
&\leq&   {C'\over 2\pi r } {|z|^\alpha\over {1+|z|^{\beta}}} \leq C  {1\over |z|}{|z|^\alpha\over {1+|z|^{\beta}}}.
\end{eqnarray*}
As a consequence, we get that $\phi^{(1)}(z)$ is in $\Psi(S^0_{\nu-\epsilon})$

By induction, we can obtain that
\begin{eqnarray*}
| \psi^{(k)}(z) |
&\leq&   C {1\over |z|^k}{|z|^\alpha\over {1+|z|^{\beta}}}.
\end{eqnarray*}
The proof of Proposition \ref{prop derivative} is complete.
\end{proof}

\begin{prop}\label{propHeat}
 Suppose that $L$ satisfies $({\bf H1})$ and $({\bf H2})$.
  Suppose $\psi\in H(S^0_{\nu})$ with
two parameters $\alpha>0, \beta>n/2 +\alpha$ such that
\begin{eqnarray}\label{q1}
|\psi(z)|\leq C{|z|^\alpha\over
1+|z|^{\beta}}.
\end{eqnarray}
Then  there exists a positive constant $C=C({n,\nu, \alpha, \beta})$
such that the kernel $K_{\psi(t {L}) }(x,y)$ of $\psi(t {L}) $
satisfies
\begin{eqnarray}\label{q2}
\big|K_{\psi(t {L})}(x,y)\big|
\leq  C\, t^{-n/2}\left(1+{|x-y|^2\over t}\right)^{-({n\over 2}+\alpha)}
\end{eqnarray}
for all $t>0$ and $x, y\in \RN$.
\end{prop}

\noindent
{\it Proof.} To prove \eqref{q2}, it suffices  to show   the following estimates:
\begin{eqnarray}\label{q3}
\big|K_{\psi(t {L})}(x,y)\big|
\leq  C\, t^{-n/2}
\end{eqnarray}
and
\begin{eqnarray}\label{q4}
\big|K_{\psi(t {L})}(x,y)\big|
\leq  C\, t^{-n/2}\left({t\over |x-y|^2}\right)^{ {n\over 2}+\alpha}.
\end{eqnarray}

Let us verify \eqref{q3}.  Note that
  for any $m\in {\Bbb N}$ and $t>0$, we have   the relationship
\begin{eqnarray}\label{q5}
  (I+tL)^{-m}={1\over  (m-1)!} \int\limits_{0}^{\infty}e^{-tsL}e^{-s} s^{m-1} ds
\end{eqnarray}
  and so when $m>n/4$,
 \begin{eqnarray*}
 \big\| (I+tL)^{-m} \big\|_{L^2(\RN)\rightarrow L^{\infty}(\RN)}\leq {1\over  (m-1)!} \int\limits_{0}^{\infty}
 \big\| e^{-tsL}\big\|_{L^2(\RN)\rightarrow L^{\infty}(\RN)} e^{-s} s^{m-1} ds\leq C t^{-n/4}
 \end{eqnarray*}
for all $t>0.$ Similarly, we have that  $ \big\| (I+tL)^{-m} \big\|_{L^1(\RN)\rightarrow L^{2}(\RN)}
= \big\| (I+tL^{\ast})^{-m} \big\|_{L^2(\RN)\rightarrow L^{\infty}(\RN)}
 \leq
C t^{-n/4}$   where $L^{\ast}$ is the adjoint operator of
$L$. Next we  choose  a constant $m$ in \eqref{q5} such that $n/4<m< (\beta-\alpha)/2$. One can write
\begin{align*}
\psi(tL)&= (I+tL)^{-2m} \big[ (I+tL)^{2m}  \psi(tL)\big] \\
&=(I+tL)^{-m} \big[ (I+tL)^{2m}  \psi(tL)\big] (I+tL)^{-m},
\end{align*} and then
\begin{eqnarray*}
 \big\|\psi(tL)\big\|_{L^1(\RN)\rightarrow L^{\infty}(\RN)}
 &\leq& \big\|(I+tL)^{-m} \big\|_{L^2(\RN)\rightarrow L^{\infty}(\RN)}\\
 & &\times \big\|(I+tL)^{2m}  \psi(tL)\big\|_{L^2(\RN)\rightarrow L^{2}(\RN)}
\big\|(I+tL)^{-m} \big\|_{L^1(\RN)\rightarrow L^{2}(\RN)}.
\end{eqnarray*}
We see that  $(1+z)^{2m}  \psi(z)\in H_{\infty}(S^0_{\nu})$  with
$
|(1+z)^{2m}  \psi(z)|\leq C{|z|^{\alpha}/(1+ |z|^{\beta-2m})}\leq C<+\infty.
$
From condition $({\bf H1})$,   $L$ has   a bounded
$H_{\infty}$-calculus on  $L^2({\mathbb R}^{n})$. This implies that
the $L^2$ operator norm of the   term $(I+tL)^{2m}  \psi(tL)$  is uniformly
bounded in $t>0$. Hence, estimate \eqref{q3} holds.

To prove  \eqref{q4}, we first  represent the operator
$\psi(tL)$ by
using the semigroup $e^{-zL}$.  As in \cite{DM}, $\psi(tL)$
(acting on $L^2({\mathbb R}^n)$) is given by
$$
\psi(tL)={1\over {2\pi i}}\int_{\Gamma}(L-\lambda
I)^{-1}\psi(t\lambda)\,d\lambda,
$$
where the contour $\Gamma=\Gamma_+\cup\Gamma_-$ is given by
$\Gamma_+(t)=te^{i\nu}$ for $t\geq 0$
and $\Gamma_-(t)=-te^{-i\nu}$ for $t<0$, and $\mu>\nu>\omega$.
For $\lambda\in \Gamma,$ substitute
$$
(L-\lambda I)^{-1}=\int_{\gamma}e^{\lambda z}e^{-zL}dz,
$$
where $\gamma$ is the ray $\{re^{i\theta}: \ 0<r<\infty\}$  with $\theta$ chosen to satisfy
 $|{\rm arg}(\lambda z)|>\pi/2$.

Changing the order of integration gives
$$
\psi(tL)=\int_{\gamma} e^{-zL}n(z)\,dz,
$$
where
$$
n(z)={1\over {2\pi i}}\int_{\Gamma}e^{\lambda z}\psi(t\lambda)\,d\lambda.
$$
Consequently, the kernel $K_{\psi(tL)}(x,y)$ of $\psi(tL)$ is given by
$$
K_{\psi(tL)}(x,y)=\int_{\gamma}p_z(x,y)n(z)\,dz.
$$
It follows from $({\bf H2})$ that
\begin{eqnarray}
|K_{\psi(tL)}(x,y)|&\leq& C\int_0^{\infty} {|z|}^{-{n\over 2}}
e^{-c{{{|x-y|^2}}\over |z|}} \left(
\int_0^{\infty}
|e^{z\lambda}\psi(t\lambda)|d|\lambda|\right)d|z|\nonumber\\
&\leq&  C\int_0^{\infty} s^{-{n\over 2}}
e^{-c{{{|x-y|^2}}\over s}} \left(
\int_0^{\infty}
 e^{-\eta sw} {(tw)^{\alpha} \over 1+ (tw)^{\beta} } dw \right) ds\nonumber
\end{eqnarray}
with $\eta>0$. Changing variables $tw\rightarrow w$ and
$s/t\rightarrow s$, we have
\begin{eqnarray*}
|K_{\psi(tL)}(x,y)|
&\leq&  Ct^{-{n\over 2}} \int_0^{\infty} s^{-{n\over 2}}
e^{-c{{{|x-y|^2}}\over st}} \left(
\int_0^{\infty}
 e^{-\eta sw} {w^{\alpha} \over 1+w^{\beta} } dw \right) ds\nonumber\\
&\leq&  Ct^{-{n\over 2}} \int_0^{\infty} s^{-({n\over 2}+\alpha +1)}
e^{-c{{{|x-y|^2}}\over st}} \left(
\int_0^{\infty}
 e^{-\eta w}  w^{\alpha}  dw \right) ds\nonumber\\
 &\leq&  Ct^{-{n\over 2}} \left({t\over |x-y|^2}\right)^{{n\over 2}+\alpha}\int_0^{\infty} s^{-({n\over 2}+\alpha +1)}
e^{-{c\over s}}  ds\nonumber\\
&\leq&  Ct^{-{n\over 2}} \left({t\over |x-y|^2}\right)^{ {n\over 2}+\alpha}.
\end{eqnarray*}
Estimate  \eqref{q4} follows readily. The proof of  Proposition~\ref{propHeat} is complete.

\medskip

\begin{prop}\label{propHeat2}
 Suppose that $L$ satisfies $({\bf H1})$,  $({\bf H2})$ and $({\bf H3})$ with some $\gamma>0$.
  Suppose $\psi\in H(S^0_{\nu})$ with
two parameters $\alpha>0, \beta>n +\alpha +{\gamma \over 2}$ such that
\begin{eqnarray}\label{q7}
|\psi(z)|\leq C{|z|^\alpha\over
1+|z|^{\beta}}.
\end{eqnarray}
Then  there exists a positive constant $C=C({n,\nu, \alpha, \beta})$
such that the kernel $K_{\psi(t {L}) }(x,y)$ of $\psi(t {L}) $
satisfies
\begin{eqnarray}\label{q8}
 \big|K_{\psi(t {L})}(x+h,y)-K_{\psi(t {L})}(x,y)\big| &+&
\big|K_{\psi(t {L})}(x, y+h)-K_{\psi(t {L})}(x,y)\big|\nonumber\\[2pt]
&\leq&  C\, \Big({|h|\over \sqrt{t}+{|x-y|}}  \Big)^{\gamma} {t^{\alpha}\over
\left(t+{|x-y|^2}\right)^{{n\over 2}+\alpha}}
\end{eqnarray}
whenever  $2|h|\leq t^{1/2}+|x-y|$, and
for all $t>0$ and $x, y\in \RN$.
\end{prop}

\begin{proof}To  prove (\ref{q8}), it suffices to
consider the part $ \big|K_{\psi(t {L})}(x+h,y)-K_{\psi(t {L})}(x,y)\big|$
since the proof of $ \big|K_{\psi(t {L})}(x, y+h)-K_{\psi(t {L})}(x,y)\big|$
is similar.

To prove this,  it suffices to verify the
following:  there exists a positive constant $C$ such that such that for all $t>0$ and $x,y,h\in {\mathbb
R}^n$,
%
\begin{eqnarray}|h|^{-\gamma} \big|K_{\psi(t {L})}(x+h,y)-K_{\psi(t {L})}(x,y)\big|
\leq C t^{-(n+\gamma)/2},
\label{q9}
\end{eqnarray}
and
\begin{eqnarray}|h|^{-\gamma} \big|K_{\psi(t {L})}(x+h,y)-K_{\psi(t {L})}(x,y)\big|
\leq C {t^{\alpha}\over {|x-y|}^{n+2\alpha+\gamma}}.
\label{q9new}
\end{eqnarray}

Let us prove \eqref{q9}. It is well known that this inequality is equivalent to
 the  boundedness
of $\psi(t {L})$ from $L^1$ to the homogeneous space $ {\dot C}^{\gamma}$
with the right hand
side  of \eqref{q9} being its operator norm.
From \eqref{q5}, we see that
when $m>(n+\gamma)/2$,
 \begin{eqnarray}\label{q10}
 \big\| (I+tL)^{-m} \big\|_{L^1(\RN)\rightarrow {\dot C}^{\gamma}(\RN)}
 &\leq& {1\over  (m-1)!} \int\limits_{0}^{\infty}
 \big\| e^{-tsL}\big\|_{L^1(\RN)\rightarrow {\dot C}^{\gamma}(\RN)} e^{-s} s^{m-1} ds\nonumber\\
  &\leq& C\int\limits_{0}^{\infty}
(ts)^{-(n+\gamma)/2} e^{-s} s^{m-1} ds\nonumber\\
 &\leq& C t^{-(n+\gamma)/2}
 \end{eqnarray}
for all $t>0.$

Fix $m$ in \eqref{q10} such that $m\in ((n+\gamma)/2, \ \beta-({n\over 2}+\alpha))$.
We see that $(1+z)^m \psi(z)\in H(S^0_{\nu})$ and
$
|(1+z)^m \psi(z)|\leq C|z|^\alpha(1+|z|)^{m-\beta}.
$
By Proposition~\ref{propHeat}, we have
\begin{eqnarray}\label{q11}
\big|K_{(1+tL)^m\psi(t {L})}(x,y)\big|
\leq  C\, t^{-n/2}\left(1+{|x-y|^2\over t}\right)^{-({n\over 2}+\alpha)}
\end{eqnarray}
for all $t>0$ and $x, y\in \RN$. From this, we have
that $\big\|(I+tL)^{m}  \psi(tL)\big\|_{L^1(\RN)\rightarrow L^{1}(\RN)}\leq C.$
Therefore,
\begin{eqnarray*}
 \big\|\psi(tL)\big\|_{L^1(\RN)\rightarrow {\dot C}^{\gamma}(\RN)}
 &\leq& \big\|(I+tL)^{-m} \big\|_{L^1(\RN)\rightarrow {\dot C}^{\gamma}(\RN)}
 \big\|(I+tL)^{m}  \psi(tL)\big\|_{L^1(\RN)\rightarrow L^{1}(\RN)}\\
 &\leq& C t^{-(n+\gamma)/2},
\end{eqnarray*}
which yields \eqref{q9}.

We now prove \eqref{q9new}. From the proof of Proposition \ref{propHeat}, we have
$$K_{\psi(tL)}(x,y)=\int_{\gamma}p_z(x,y)n(z)\,dz,
$$
which implies that
\begin{align*}
&\big|K_{\psi(t {L})}(x+h,y)-K_{\psi(t {L})}(x,y)\big|=\bigg|\int_0^{\infty}[p_z(x+h,y)-p_z(x,y)]n(z)\,dz\bigg|\\
&\leq  C\int_0^{\infty} \Big( {|h|\over \sqrt{s}}\Big)^\gamma s^{-{n\over 2}}
e^{-c{{{|x-y|^2}}\over s}} \left(
\int_0^{\infty}
 e^{-\eta sw} {(tw)^{\alpha} \over 1+ (tw)^{\beta} } dw \right) ds\\
&\leq C    {|h|^\gamma \over t^{n+\gamma \over 2}  } \int_0^\infty s^{-({n\over 2}+\frac{\gamma}{2}+\alpha+1)} e^{-c{|x-y|^2\over st}} \bigg(\int_0^\infty e^{-\eta w} w^\alpha dw \bigg)\, ds\\
&\leq C    {|h|^\gamma\over t^{{n+\gamma \over 2} } }  \Big( {t\over |x-y|^2} \Big)^{{n\over 2}+{\gamma \over 2}+\alpha}\int_0^\infty s^{-({n+\gamma \over 2}+\alpha+1)} e^{-{c\over s}}\,ds\\
&\leq C    {\Big( {|h|\over |x-y|} \Big)^\gamma}   { t^\alpha\over (|x-y|^2)^{{n\over 2}+\alpha} }.
\end{align*}

The proof of Proposition~\ref{propHeat2}  is end.
\end{proof}

Recall that  $\zeta\in H(S^0_{\nu})$ with
two parameters $\alpha>0, \beta>n +\alpha+3+\gamma$ such that $\zeta$ satisfies \eqref{q1}, and $q(\lambda)=\lambda^2\zeta^2(\lambda)$.
Set $ \varphi(\lambda)=\lambda^2\zeta(\lambda).$ Then $q(\lambda)=\zeta(\lambda)\varphi(\lambda)$.
Denote $q_t(x,y)$ the kernel of the operator $q(tL)$ where $t>0$.

\begin{prop}\label{s-prop3.3} Suppose that $L$ satisfies $({\bf H1})-({\bf H3})$. Suppose that $f\in L^p(\RN), \ 1\leq p<\infty$. For any $t>0 $, $s>0 $, and $x, y\in \RN$, the following results hold.

\begin{eqnarray}\label{s-e2.1}
\big|K_{q(t^2{L}) q(s^2L)}(x,y)\big|
\leq  C\, \Big({t\over s}\wedge {s\over t}\Big) {(t+s)^{\alpha}\over (t+s+ |x-y|)^{n+\alpha}},
\end{eqnarray}

\begin{eqnarray}\label{s-e2.1.1}
\big|K_{ (t^2Lq'(t^2{L})) q(s^2L)}(x,y)\big|
\leq  C\, \Big({t\over s}\wedge {s\over t}\Big) {(t+s)^{\alpha}\over (t+s+ |x-y|)^{n+\alpha}},
\end{eqnarray}
where we use $a\wedge b$ to denote $\min\{a,b\}$ for every positive numbers $a$ and $b$.

\begin{eqnarray}\label{s-e2.2}
\big| q(t^2{L})f(x)-q(t^2L)f(y)\big|
\leq C \Big(\frac{|x-y|}{t}\Big)^{\gamma}\inf_{|u-x|<t}\mathcal{M}\big(\varphi(t^2L)f\big)(u),
\end{eqnarray}
when $|x-y|<t/2$, where we use $\mathcal{M}$ to denote the Hardy--Littlewood maximal operator.

\begin{eqnarray}\label{s-e2.3}
\big| q(s_1^2{L})f(x)-q(s_2^2L)f(x)\big|
\leq  2\Big(\ln \frac{s_2}{s_1}\Big)^{1/2} \Big(\int_{s_1}^{s_2} |r^2Lq^{(1)}(r^2L)f(x)|^2 \,\frac{dr}{r}\Big)^{1/2},
\end{eqnarray}
when $0<s_1\leq s_2.$

\end{prop}

\begin{proof}
To prove (\ref{s-e2.1}), it suffices to prove that if $s\leq t$, then

\begin{align}\label{s-e2.4}
\big|K_{q(t^2{L}) q(s^2L)}(x,y)\big|
\leq  C\, \Big({s\over t}\Big) {t^{\alpha}\over (t+ |x-y|)^{n+\alpha}};
\end{align}

In fact,

$$
q(s^2L)q(t^2L)=s^4L^2\zeta^2(s^2L)t^4L^2\zeta^2(t^2L)=\Big(\frac{s^4}{t^4}\Big) \zeta^2(s^2L) t^8L^4\zeta^2(t^2L).
$$
Note that $\zeta^2(x)$ and $x^4\zeta^2(x)$ satisfy the condition of Proposition \ref{propHeat}.
Then by Proposition \ref{propHeat}, we obtain
\begin{align*}
|K_{q(s^2L)q(t^2L)}(x,y)|&\leq \Big(\frac{s^4}{t^4}\Big) \int_{\RN}\big|K_{\zeta^2(s^2L)}(x,u)\big| \big|K_{t^8L^4\zeta^2(t^2L)}(u,y) \big| du\\
&\leq C \Big(\frac{s^4}{t^4}\Big) \int_{\RN}\frac{s^\alpha}{(s+|x-u|)^{n+\alpha}}\frac{t^\alpha}{(t+|u-y|)^{n+\alpha}} du\\
&\leq C\Big(\frac{s}{t}\Big) \frac{(s+t)^\alpha}{(s+t+|x-y|)^{n+\alpha}},
\end{align*}
which have proved (\ref{s-e2.4}), and hence \eqref{s-e2.1} follows.

Similarly we can obtain \eqref{s-e2.1.1} by the argument above and Propositions \ref{prop derivative} and \ref{propHeat}.

Next, let us prove (\ref{s-e2.2}). One can write

$$
q(t^2L)f(x)=\int_{\RN} K_{\zeta(t^2L)}(x,z) \varphi(t^2L)f(z)\, dz.
$$
Hence, if $|x-y|\leq t/2$, we use Proposition \ref{propHeat2} to obtain
\begin{align*}
|q(t^2L)f(x)&-q(t^2L)f(y)|\\
&\leq \int_{\RN} |K_{\zeta(t^2L)}(x,z)-K_{\zeta(t^2L)}(y,z)| |(t^2L)^2 \zeta(t^2L)f(z)|\, dz\\
&\leq \int_{\RN} \Big(\frac{|x-y|}{t} \Big)^{\gamma} \frac{t^\alpha}{(t+|x-z|)^{n+\alpha}} |\varphi(t^2L)f(z)|\, dz\\
&\leq  C \Big(\frac{|x-y|}{t}\Big)^{\gamma}\inf_{|u-x|<t}\mathcal{M}\big(\varphi(t^2L)f\big)(u).
\end{align*}

Let us prove (\ref{s-e2.3}). Observe that
\begin{align*}
 q(s_2^2{L})f(x)-q(s_1^2L)f(x)&=\int_{s_1^2}^{s_2^2} \frac{d}{dr}q(rL)f(x) \, dr\\
&=\int_{s_1^2}^{s_2^2} rL q^{(1)}(rL)f(x) \frac{dr}{r}\\
&=2 \int_{s_1}^{s_2}  r^2L q^{(1)}(r^2L)f(x) \frac{dr}{r}
\end{align*}
Hence, we use H\"older's inequality to obtain

\begin{align*}
\Big| q(t^2{L})f(x)-q(s^2L)f(x)\Big|\leq 2\Big(\ln \frac{t}{s}\Big)^{1/2}\Big(\int_s^t  \Big|r^2L q^{(1)}(r^2L)f(x)\Big|^2 \frac{dr}{r}\Big)^{1/2}.
\end{align*}

\end{proof}

\begin{remark}\label{s-remark1}
For any complex function  $\eta(\lambda)$, we denote $\bar{\eta}(\lambda):=\overline{\eta\big(\bar{\lambda}\big)}$. Recall that  $\zeta\in H(S^0_{\nu})$ with
two parameters $\alpha>0, \beta>n +\alpha+3+\gamma$ such that $\zeta$ satisfies \eqref{q1}, and $q(\lambda)=\lambda^2\zeta^2(\lambda)$. Then $\bar{q}(\lambda)=\lambda^2\bar{\zeta}^2(\lambda)$, and $\bar{\zeta}(\lambda)$ satisfy the same condition as that of $\zeta$.
We remark that the above estimates \eqref{s-e2.1}--\eqref{s-e2.3} also hold when we replace $L$ and $q$ by $L^*$ and $\bar{q}$ respectively, since when
$L$ satisfies $({\bf H1})-({\bf H3})$, $L^*$ also satisfies $({\bf H1})-({\bf H3})$ and   $\bar{q}$ satisfy the same conditions as that of  $q$.
\end{remark}

Next we recall the Littlewood--Paley theory as follows. For the proof, we refer to Theorem 6 of  \cite{AuDM}, see also (3.8) of \cite{DY2}, as well as \cite{Yan2}.
\begin{lemma}\label{g fucntion Lp}
Suppose that $L$ satisfies {\bf (H1)-(H2)}. Let $\psi\in H(S_\mu^0)$  and  there exist $\alpha>0$ and $\beta>\alpha$, such that
$$
|\psi(z)|\leq C\frac{|z|^\alpha}{1+|z|^\beta},  \quad  {\rm for \ any} \ z\in S_\mu^0.
$$
Then for any $1<p<+\infty$, there exists  constants $C_p$, such that
\begin{align*}
\Big\|\Big(\int_0^\infty
\big|\psi(t^2L)f\big|^2\frac{dt}{t}\Big)^{1/2}\Big\|_p
\leq C_p\big\|f\big\|_p.
\end{align*}
\end{lemma}

\vskip 0.5 true cm

\section{Boundedness of the operator $T_\delta$ on $L^2(\mathbb{R}^n)$}

 \setcounter{equation}{0}

Let
$q(z)$ be the same as in \eqref{varphi}.
%
%
Recall that  for any $f\in L^2(\RN)$, the operator $T_\delta$ in \eqref{Tdelta} is defined as
\begin{align*}
T_ \delta(f)(x):=\ln\delta\sum\limits_j\sum\limits_{\tau\in I_j}|Q_{\tau}^j| q_j(x,y_{Q_{\tau}^j})q_{\delta^{-2j}}(L)(f)(y_{Q_{\tau}^j}),
\end{align*}
where $y_{Q_{\tau}^j}$ is any point in the cube $Q_{\tau}^j$.

The main aim in this section is to show that $T_ \delta(f)(x)$ is well-defined and bounded on $L^2(\RN)$.

First, we point out that from Proposition \ref{propHeat2}, $q_j(x,y)$  is continuous in both $x$ and $y$. Hence, we see that $q_j(x,y_{Q_{\tau}^j})$ is well-defined for
any $y_{Q_{\tau}^j}$ in the cube $Q_{\tau}^j$.

Second, we consider the term $q_{\delta^{-2j}}(L)(f)(y_{Q_{\tau}^j})$, which is defined as
\begin{align}\label{se3.qjxy}
q_{\delta^{-2j}}(L)(f)(y_{Q_{\tau}^j}) = \int_{\RN} q_j(y_{Q_{\tau}^j},y)f(y)dy,\quad f\in L^2(\RN).
\end{align}
We point out that for every $f\in L^2(\RN)$, $q_{\delta^{-2j}}(L)(f)(y_{Q_{\tau}^j})$ is well-define for every
$y_{Q_{\tau}^j}\in Q_{\tau}^j$. In fact, since $q_{\delta^{-2j}}(L)$ is a bounded operator on $L^2(\RN)$, we get that
for $f\in L^2(\RN)$, $q_{\delta^{-2j}}(L)(f)$ is also in $L^2(\RN)$, and hence
$q_{\delta^{-2j}}(L)(f)(x)$ is defined for a.e. $x\in \RN$. Moreover, since $q_j(x,y)$
is continuous in $x$ and satisfies \eqref{q8}, we see that
$q_{\delta^{-2j}}(L)(f)(x)$ is also continuous in $x$. Hence $q_{\delta^{-2j}}(L)(f)(y_{Q_{\tau}^j})$ is well-defined for every
$y_{Q_{\tau}^j} \in Q_{\tau}^j$.

\begin{theorem}\label{Thm Tdelta}
Let all the notation be the same as above. We have that $T_\delta$ is well-defined and bounded on $L^2(\mathbb R^n)$.
\end{theorem}

Before proving this theorem, we first establish the following Littlewood--Paley estimate on $L^2(\RN)$.
\begin{lemma}\label{le L-P} There exists a positive constant $C$ such that for every $1<\delta<2$,
\begin{eqnarray}\label{e1 lemma L2 to L2}
\bigg\|\sqrt{\ln\delta}\Big(\sum\limits_j\sum\limits_{\tau\in I_j}\big|q_{\delta^{-2j}}(L)f(y_{Q_{\tau}^j})\big|^2\chi_{Q_{\tau}^j}(x)\Big)^{1/2}\bigg\|_{L^2(\mathbb{R}^n)}\leq
C_p\big\|f\big\|_{L^2(\mathbb{R}^n)}.
\end{eqnarray}
 \end{lemma}

\begin{proof}

First, we need a Calder\'on type reproducing formula, which is inspired from the
$H_\infty$-calculus for $L$. We start from the following fact: for $q(z)$ defined as in \eqref{varphi},
$$ {1\over2}\int_0^\infty q(t)\cdot q(t) {dt\over t}=:c,$$
it is direct to see that $c\not=0$.

Then,
by $H_\infty$-functional
calculus (\cite{Mc}), for every $f\in L^2({\Bbb R}^n)$,
%
\begin{eqnarray}\label{e3.1}
f&=&c^{-1}\int_0^\infty q_{t^2}(L)q_{t^2}(L)f\frac{dt}{t}
\end{eqnarray}
in the sense of $L^2(\RN)$. To be more precise, we have
\begin{eqnarray}\label{e3.1 prime}
f=\lim_{N\to\infty} F_N \quad{\rm\ in\ the\ sense\ of }\ L^2(\RN), {\rm\ where}\quad F_N:=c^{-1}\int_{N^{-1}}^N q_{t^2}(L)q_{t^2}(L)f\frac{dt}{t}.
\end{eqnarray}

 Then for any fixed $1< \delta<2$, from the reproducing formula \eqref{e3.1}, \eqref{se3.qjxy} and the fact that $q_{\delta^{-2j}}(L)f(x)$ is a continuous function (see the explanation below \eqref{se3.qjxy}), we have that
\begin{align}\label{se3.1 new}
q_{\delta^{-2j}}(L)f(y_{Q_{\tau}^j})&=q_{\delta^{-2j}}(L)\, \bigg(c^{-1}\int_0^\infty
q_{t^2}(L)q_{t^2}(L)f\frac{dt}{t}\bigg)\,(y_{Q_{\tau}^j}),
\end{align}
where $y_{Q_\tau^j}$ is any point in the cube $Q_\tau^j$.

 Next, from \eqref{se3.1 new} and \eqref{e3.1 prime},
 by noting that $q_{\delta^{-2j}}(L)$ is a bounded, linear operator on $L^2(\RN)$, we have that
\begin{align}\label{se3.1 final}
q_{\delta^{-2j}}(L)f(y_{Q_{\tau}^j})
&=q_{\delta^{-2j}}(L)\, \bigg(\lim_{N\to\infty} F_N\bigg)\,(y_{Q_{\tau}^j})\nonumber\\
&=\lim_{N\to\infty} q_{\delta^{-2j}}(L) \big(F_N\big)\,(y_{Q_{\tau}^j})\nonumber\\
&=\lim_{N\to\infty}  c^{-1}\int_{N^{-1}}^N q_{\delta^{-2j}}(L) q_{t^2}(L)q_{t^2}(L)f(y_{Q_{\tau}^j})\frac{dt}{t}\nonumber\\
&= c^{-1}\int_{0}^\infty q_{\delta^{-2j}}(L) q_{t^2}(L)q_{t^2}(L)f(y_{Q_{\tau}^j})\frac{dt}t,
\end{align}
where the third equality follows from the size estimate of the kernels of $q_{\delta^{-2j}}(L)$ and
$q_{t^2}(L)q_{t^2}(L)$ (see Proposition \ref{s-prop3.3}) and  Fubini's theorem.

Note that from the almost orthogonality estimates in Section 2 (Proposition  \ref{s-prop3.3}), we have
$$|q_{\delta^{-2j}}(L)q_{t^2}(L)(x,z)|\leq C
\big(\frac{ \delta^{-j}}{t}\big)\wedge\big(\frac{t}{ \delta^{-j}}\big)\frac{(t+ \delta^{-j})}{(t+ \delta^{-j}+|x-z|)^{n+1}},
$$
where $q_{\delta^{-2j}}(L)q_{t^2}(L)(x,z)$ is the kernel of $q_{\delta^{-2j}}(L)q_{t^2}(L)$.

Hence
\begin{align}\label{e2.1}
|q_{\delta^{-2j}}(L)q_{t^2}(L)q_{t^2}(L)f(y_{Q_{\tau}^j})|&=\bigg| \int_{\RN} q_{\delta^{-2j}}(L)q_{t^2}(L)(x,z) \,q_{t^2}(L)f(z) dz\bigg|\nonumber\\
&\leq
C\int_{\RN}\big(\frac{ \delta^{-j}}{t}\big)\wedge\big(\frac{t}{ \delta^{-j}}\big)
\frac{(t+ \delta^{-j})}{(t+ \delta^{-j}+|y_{Q_{\tau}^j}-z|)^{n+1}}|q_{t^2}(L)f(z)|dz\nonumber\\
&\leq
C\big(\frac{ \delta^{-j}}{t}\big)\wedge\big(\frac{t}{ \delta^{-j}}\big)\inf\limits_{y\in
Q_{\tau}^j}\mathcal{M}(q_{t^2}(L)f)(y).
\end{align}

\noindent By substituting (\ref{e2.1}) into (\ref{se3.1 final}), we have

\begin{eqnarray*}
|q_{\delta^{-2j}}(L)f(y_{Q_{\tau}^j})|\leq
C\int_0^\infty\big(\frac{ \delta^{-j}}{t}\big)\wedge\big(\frac{t}{ \delta^{-j}}\big)\inf\limits_{y\in
Q_{\tau}^j}\mathcal{M}(q_{t^2}(L)f)(y)\frac{dt}{t}.
\end{eqnarray*}

\noindent Observe that

\begin{eqnarray}\label{se3.2}
\int _0^\infty
\big(\frac{ \delta^{-j}}{t}\big)\wedge\big(\frac{t}{ \delta^{-j}}\big)\frac{dt}{t}=2,
\qquad \sum\limits_j
\big(\frac{ \delta^{-j}}{t}\big)\wedge\big(\frac{t}{ \delta^{-j}}\big)\leq
\frac{2{\delta}}{ \delta-1}.
\end{eqnarray}

\noindent We then apply H\"older's inequality, Lebesgue's theorem
and (\ref{se3.2}) to obtain
\begin{align}\label{key0}
&\ln\delta\sum\limits_j\sum\limits_{\tau\in I_j}|q_{\delta^{-2j}}(L)f(y_{Q_{\tau}^j})|^2\chi_{Q_{\tau}^j}(x)\\
&\leq C\ln\delta\sum\limits_j\sum\limits_{\tau\in
I_j}\Big(\int_0^\infty\big(\frac{ \delta^{-j}}{t}\big)\wedge\big(\frac{t}{ \delta^{-j}}\big)\inf\limits_{y\in
Q_{\tau}^j}\mathcal{M}(q_{t^2}(L)f)(y)\frac{dt}{t}\Big)^2\chi_{Q_{\tau}^j}(x)\nonumber\\
&\leq C\ln\delta\sum\limits_j\sum\limits_{\tau\in
I_j}\int_0^\infty\big(\frac{ \delta^{-j}}{t}\big)\wedge\big(\frac{t}{ \delta^{-j}}\big)
\big(\mathcal{M}(q_{t^2}(L)f)(x)\big)^2\frac{dt}{t}\chi_{Q_{\tau}^j}(x)\nonumber\\
&\leq C\ln\delta
\sum\limits_j\int_0^\infty\big(\frac{ \delta^{-j}}{t}\big)\wedge\big(\frac{t}{ \delta^{-j}}\big)
\big(\mathcal{M}(q_{t^2}(L)f)(x)\big)^2\frac{dt}{t}\nonumber\\
&\leq C\ln\delta\frac{ \delta}{ \delta-1}\int_0^\infty
\big(\mathcal{M}(q_{t^2}(L)f)(x)\big)^2\frac{dt}{t}\nonumber\\
&\leq C\int_0^\infty
\big(\mathcal{M}(q_{t^2}(L)f)(x)\big)^2\frac{dt}{t}.\nonumber
\end{align}

\noindent Therefore, it follows from the $L^2$-boundedness of the Hardy--Littlewood maximal operator and Lemma \ref{g fucntion Lp} that
\begin{align*}
&\bigg\|\sqrt{\ln\delta}\Big(\sum\limits_j\sum\limits_{\tau\in I_j}\big|q_{\delta^{-2j}}(L)f(y_{Q_{\tau}^j})\big|^2\chi_{Q_{\tau}^j}(x)\Big)^{1/2}\bigg\|_{L^2(\mathbb{R}^n)}\\
&\leq
C\Big\|\Big(\int_0^\infty
\big(\mathcal{M}(q_{t^2}(L)f)\big)^2\frac{dt}{t}\Big)^{1/2}\Big\|_{L^2(\mathbb{R}^n)} \leq C\Big\|\Big(\int_0^\infty
\big|q_{t^2}(L)f\big|^2\frac{dt}{t}\Big)^{1/2}\Big\|_{L^2(\mathbb{R}^n)}
\leq C\big\|f\big\|_{L^2(\mathbb{R}^n)},
\end{align*}

\noindent which shows that (\ref{e1 lemma L2 to L2}) holds.
\end{proof}

We now start to provide the proof for Theorem \ref{Thm Tdelta}.

\begin{proof}[Proof of Theorem \ref{Thm Tdelta}]
Let $\Lambda_{finite}$ be an arbitrary finite subset of the integers $\mathbb Z$.
For every $j\in\mathbb Z$, let $I_{j,finite}$ be an arbitrary finite subset of the index
$I_j$.

For every $f\in L^2(\RN)$, we consider the following  auxiliary operator
\begin{align}
T_{\delta,j,\tau}(f)(x):=\ln\delta\ |Q_{\tau}^j|\ q_j(x,y_{Q_{\tau}^j})\ q_{\delta^{-2j}}(L)(f)(y_{Q_{\tau}^j}).
\end{align}
First, it is easy to see that for every $h\in L^2(\RN)$,
\begin{align*}
\left\langle \sum_{j\in\Lambda_{finite}}\sum_{\tau\in I_{j,finite}} T_{\delta,j,\tau}(f),h\right\rangle&=\left\langle\ln\delta \sum_{j\in\Lambda_{finite}}\sum_{\tau\in I_{j,finite}} \ |Q_{\tau}^j|\ q_{\delta^{-2j}}(L)(\cdot,y_{Q_{\tau}^j})\ q_{\delta^{-2j}}(L)(f)(y_{Q_{\tau}^j}),h(\cdot)\right\rangle\\
&=\ln\delta \sum_{j\in\Lambda_{finite}}\sum_{\tau\in I_{j,finite}} \ |Q_{\tau}^j|\ q_{\delta^{-2j}}(L)(h)(y_{Q_{\tau}^j})\ q_{\delta^{-2j}}(L)(f)(y_{Q_{\tau}^j}),
\end{align*}
where the last equality follows from the fact that the sums on $j$ and $\tau$ are finite.
Then, by using H\"older's inequality, we obtain that
\begin{align*}
&\left|\left\langle \sum_{j\in\Lambda_{finite}}\sum_{\tau\in I_{j,finite}} T_{\delta,j,\tau}(f),h\right\rangle\right|\\
&\leq\ln\delta \bigg(\sum_{j\in\Lambda_{finite}}\sum_{\tau\in I_{j,finite}} \ |Q_{\tau}^j|\ |q_{\delta^{-2j}}(L)(h)(y_{Q_{\tau}^j})|^2 \bigg)^{1\over2}\ \bigg(\sum_{j\in\Lambda_{finite}}\sum_{\tau\in I_{j,finite}} \ |Q_{\tau}^j|\ |q_{\delta^{-2j}}(L)(f)(y_{Q_{\tau}^j})|^2 \bigg)^{1\over2}\\
&\leq \bigg\|\sqrt{\ln\delta}\bigg(\sum_{j\in\Lambda_{finite}}\sum_{\tau\in I_{j,finite}} \ |q_{\delta^{-2j}}(L)(h)(y_{Q_{\tau}^j})|^2\chi_{Q_{\tau}^j}(\cdot) \bigg)^{1\over2}\bigg\|_{L^2(\RN)}\\
&\qquad\times \bigg\|\sqrt{\ln\delta}\bigg(\sum_{j\in\Lambda_{finite}}\sum_{\tau\in I_{j,finite}} \ |q_{\delta^{-2j}}(L)(f)(y_{Q_{\tau}^j})|^2\chi_{Q_{\tau}^j}(\cdot) \bigg)^{1\over2}\bigg\|_{L^2(\RN)}\\
&\leq C \|f\|_{L^2(\RN)}\|h\|_{L^2(\RN)},
\end{align*}
where the last inequality follows from Lemma \ref{le L-P}, and hence it is clear that the constant $C$ is independent of $\delta$, $f$, $h$, $\Lambda_{finite}$ and $I_{j,finite}$.

This implies that
\begin{align}\label{key1}
\bigg\| \sum_{j\in\Lambda_{finite}}\sum_{\tau\in I_{j,finite}} T_{\delta,j,\tau}(f)\bigg\|_{L^2(\RN)}\leq C \|f\|_{L^2(\RN)}.
\end{align}

Next we use the Rademacher functions $r_{j}$ of Appendix C.1 in \cite{G}. These functions are defined for nonnegative integers $j$, but we now reindex them so that the subscript are represented by $\{j,\tau\}$, where $j\in \mathbb Z$ and $\tau\in I_j$. The fundamental property of these functions
is their orthogonality, that is
$$ \int_0^1 r_{j,\tau} (\omega) r_{j',\tau'}(\omega) d\omega =0$$
when $j\not=j'$ or $\tau\not=\tau'$.

Now for every $\Lambda_{finite}$ and $I_{j,finite}$ and for every $f\in L^2(\RN)$,
we obtain that
\begin{align}\label{Rade}
&\int_0^1 \bigg\| \sum_{j\in\Lambda_{finite}}\sum_{\tau\in I_{j,finite}} r_{j,\tau}(\omega)\ T_{\delta,j,\tau} (f)\bigg\|_{L^2(\RN)}^2d\omega\\
&= \sum_{j\in\Lambda_{finite},\ \tau\in I_{j,finite}}   \bigg\| \ T_{\delta,j,\tau} (f)\bigg\|_{L^2(\RN)}^2\nonumber\\
&\quad +   \int_0^1  \sum_{j\in\Lambda_{finite},\ \tau\in I_{j,finite}} \  \sum_{\substack{j'\in\Lambda_{finite},\ \tau'\in I_{j',finite}\\ (j,\tau)\not=(j',\tau') }}  r_{j,\tau}(\omega) r_{j',\tau'}(\omega)\ \langle T_{\delta,j,\tau} (f),T_{\delta,j',\tau'} (f)\rangle\ d\omega    \nonumber\\
&= \sum_{j\in\Lambda_{finite},\ \tau\in I_{j,finite}}   \bigg\| \ T_{\delta,j,\tau} (f)\bigg\|_{L^2(\RN)}^2\nonumber.
\end{align}

For any fixed $\omega\in[0,1]$ we now repeat the proof of \eqref{key1} for the operators $r_{j,\tau}(\omega)\ T_{\delta,j,\tau}$, and use the fact that $ r_{j,\tau}(\omega)=\pm1 $ to obtain that
\begin{align*}
\left|\left\langle \sum_{j\in\Lambda_{finite}}\sum_{\tau\in I_{j,finite}} r_{j,\tau}(\omega)T_{\delta,j,\tau}(f),h\right\rangle\right|
&\leq C \|f\|_{L^2(\RN)}\|h\|_{L^2(\RN)}, \quad\forall h\in L^2(\RN),
\end{align*}
which implies that
\begin{align}\label{key2}
\bigg\| \sum_{j\in\Lambda_{finite}}\sum_{\tau\in I_{j,finite}} r_{j,\tau}(\omega)T_{\delta,j,\tau}(f)\bigg\|_{L^2(\RN)}\leq C \|f\|_{L^2(\RN)}.
\end{align}

Combining the estimates of \eqref{Rade} and \eqref{key2}, we get that
\begin{align*}
\sum_{j\in\Lambda_{finite},\ \tau\in I_{j,finite}}   \bigg\| \ T_{\delta,j,\tau} (f)\bigg\|_{L^2(\RN)}^2
&=\int_0^1 \bigg\| \sum_{j\in\Lambda_{finite}}\sum_{\tau\in I_{j,finite}} r_{j,\tau}(\omega)\ T_{\delta,j,\tau} (f)\bigg\|_{L^2(\RN)}^2d\omega\\
&\leq  C\int_0^1 \| f\|_{L^2(\RN)}^2d\omega = C \| f\|_{L^2(\RN)}^2.
\end{align*}

By taking the limit of $\tau$ and $j$, we obtain that
\begin{align}\label{key3}
\sum_{j\in \mathbb Z,\ \tau\in I_{j}}   \bigg\| \ T_{\delta,j,\tau} (f)\bigg\|_{L^2(\RN)}^2
\leq C \| f\|_{L^2(\RN)}^2.
\end{align}

Next, we show that for every $f\in L^2(\RN)$, the sequence
$$ \bigg\{ \sum_{j=-N}^N \sum_{\tau=0}^N T_{\delta,j,\tau} (f) \bigg\} $$
is a Cauchy sequence in $L^2(\RN)$. Suppose that this is not the case. This means that there is some $\epsilon>0$ and a subsequence of integers $1<N_1<N_2<N_3<\cdots$ such that
\begin{align}\label{T tilde}
\big\| \tilde T_{\delta,k}(f) \big\|_{L^2(\RN)}\geq\epsilon,
\end{align}
where
$$ \tilde T_{\delta,k}(f):= \sum_{j=-N_{k+1}}^{N_{k+1}} \sum_{\tau=0}^{N_{k+1}} T_{\delta,j,\tau} (f)  - \sum_{j=-N_{k}}^{N_{k}} \sum_{\tau=0}^{N_{k}} T_{\delta,j,\tau} (f). $$

For any fixed $\omega\in[0,1]$, we repeat the proof of \eqref{key1} to the operator $r_{k}(\omega)\tilde T_{\delta,k}$ to obtain that
\begin{align*}
\bigg\| \sum_{k=1}^N r_k(\omega) \tilde T_{\delta,k} (f) \bigg\|_{L^2(\RN)}
\leq C\|f\|_{L^2(\RN)}.
\end{align*}
Squaring and integrating this inequality with respect to $\omega\in[0,1]$, and using \eqref{Rade}
with $\tilde T_{\delta,k} $ in the place of $T_{\delta,j,\tau}$ and $k\in \{ 1,2,\ldots, K \}$ in the place of
${j\in\Lambda_{finite},\ \tau\in I_{j,finite}} $, we get that
$$  \sum_{k=1}^K \big\| \tilde T_{\delta,k}(f) \big\|_{L^2(\RN)}^2 \leq C\|f\|_{L^2(\RN)}^2.  $$
But this contradicts \eqref{T tilde} as $K\to\infty$.

So we conclude that every sequence
$$ \bigg\{ \sum_{j=-N}^N \sum_{\tau=0}^N T_{\delta,j,\tau} (f) \bigg\} $$
is a Cauchy sequence in $L^2(\RN)$, and thus it converges to
$ T_\delta $.

This, together with \eqref{key1}, implies   that $T_\delta$ is a bounded operator on $L^2(\RN)$ with norm at most some constant $C$.
\end{proof}

\vskip 0.5 true cm

\section{Frame decompositions on $L^p(\mathbb{R}^n)$, $1<p<\infty$}

 \setcounter{equation}{0}

Suppose $\zeta\in H(S^0_{\nu})$ with
two parameters $\alpha>0, \beta>n +\alpha+3+\gamma$ such that $\zeta$ satisfies \eqref{q1}. For the sake of simplicity,
in the rest of the paper, we take $\alpha=\gamma=1$.
Recall that
$q(z)=z^2\zeta^2(z).$ Similar to Section 3, we denote $q_t(x,y)$ the kernel of the operator $q(tL)$, where $t>0$, and denote $$q_j(x,y):=q_{\delta^{-2j}}(x,y),$$ where $j\in\mathbb{Z}$.

\subsection{Littlewood--Paley $g$ functions on $L^p(\mathbb{R}^n)$, $1<p<\infty$ }

We introduce four discrete Littlewood--Paley $g$-functions.
 For any fixed
$\delta>1$, we define
\begin{eqnarray*}
g_{1,\delta}(L)f(x)&:=&\sqrt{\ln\delta}\Big(\sum\limits_j\sum\limits_{\tau\in I_j}\big|q_{\delta^{-2j}}(L)f(y_{Q_{\tau}^j})\big|^2\chi_{Q_{\tau}^j}(x)\Big)^{1/2};\\
g_{2,\delta}(L)f(x)&:=&\Big(\sum\limits_j\sum\limits_{\tau\in
I_j}\int_{\delta^{-j}}^{\delta^{-j+1}}\big|q_{t^2}(L)f(y_{Q_{\tau}^j})\big|^2
\frac{dt}{t}\chi_{Q_{\tau}^j}(x)\Big)^{1/2};\\
g_{3,\delta}(L)f(x)&:=&\Big(\sum_j \sum\limits_{\tau\in I_j}
\int_{\delta^{-j}}^{\delta^{-j+1}}\int_{Q_{\tau}^j}\big|q_{t^2}(L)f(y)\big|^2dy\frac{dt}{t}
\ {1\over |Q_{\tau}^j|}\chi_{Q_{\tau}^j}(x) \Big)^{1\over 2};\\
g_{4, \delta}(L)f(x) &:=&\Big(\sum\limits_j\sum\limits_{\tau\in I_j}
\int_{ \delta^{-j}}^{ \delta^{-j+1}}| s^2Lq^{(1)}(s^2L)  f(y_{Q_{\tau}^j})|^2\frac{ds}{s}
\chi_{Q_{\tau}^j}(x)\Big)^{1/2}.
\end{eqnarray*}

\noindent

\begin{lemma}\label{le2.1} Suppose $1<p<+\infty$. There exists a positive constant $C_p$ such that for every $1<\delta<2$,
\begin{eqnarray}\label{e1 lemma Lp to Lp}
\big\|g_{i,\delta}(L)f\big\|_{L^p(\mathbb{R}^n)}\leq
C_p\big\|f\big\|_{L^p(\mathbb{R}^n)},
\end{eqnarray}
where $i=1,2,3,4$.
 \end{lemma}

\begin{proof}

We first estimate $g_{1, \delta}(f)$.
Following the same estimate as in the proof of Lemma \ref{le L-P} we can obtain that
\begin{eqnarray*}
\big(g_{1, \delta}(L)f(x)\big)^2
\leq C\int_0^\infty
\big(\mathcal{M}(q_{t^2}(L)f)(x)\big)^2\frac{dt}{t}.
\end{eqnarray*}

\noindent Therefore, it follows from the vector-value maximal
theorem (see Proposition 4.5.11, \cite{G}) and Lemma \ref{g fucntion Lp} that
\begin{align*}
\Big\|g_{1, \delta}(L)f\Big\|_p&\leq
C\Big\|\Big(\int_0^\infty
\big(\mathcal{M}(q_{t^2}(L)f)\big)^2\frac{dt}{t}\Big)^{1/2}\Big\|_p \leq C\Big\|\Big(\int_0^\infty
\big|q_{t^2}(L)f\big|^2\frac{dt}{t}\Big)^{1/2}\Big\|_p
\leq C\big\|f\big\|_p,
\end{align*}

\noindent which shows that (\ref{e1 lemma Lp to Lp}) holds for $i=1$.

The proofs of \eqref{e1 lemma Lp to Lp} for $i=2,3$
 are similar to that for $i=1$. The proof of \eqref{e1 lemma Lp to Lp} for $i=4$ is similar to that for $i=1$, but via the almost orthogonality estimate \eqref{s-e2.1.1} in Proposition \ref{s-prop3.3}. We omit the details here.
\end{proof}

\begin{remark}\label{remark g function Lp}
For  $i=1,2,3,4$,   we define $g^*_{i,\delta}(L^*)$ following the same way as $g_{i,\delta}(L)$ with $L$ and $q$ replaced by $L^*$ and $\bar{q}$, respectively.
Then, by Remark \ref{s-remark1},  the above estimates \eqref{e1 lemma Lp to Lp}  also hold for  $g^*_{i,\delta}(L^*)$ for $i=1,2,3,4$.
 \end{remark}

\subsection{Proof of frame decomposition on Lebesgue spaces}

In the next two results, we obtain estimates on the norms of the operators $T_{\delta}$ and $I - T_{\delta}$ on Lebesgue spaces.

\begin{theorem}\label{th2.1}

For every  $1<p<\infty$, there exists a constant $C_p>0$ such that
$$
\|T_ \delta(f)\|_p\leq C_p \|f\|_p.
$$
\end{theorem}
\begin{proof}

To see this, we first recall
from Theorem \ref{Thm Tdelta}, $T_\delta$ is well-defined and bounded on $L^2(\RN)$.
Hence, for every
$f\in L^p(\RN)\cap L^2(\RN)$ and $h\in L^{p'}(\RN)\cap L^2(\RN)$,  from the definition of $T_\delta$ as in \eqref{Tdelta}, we have
\begin{eqnarray*}
\langle T_ \delta(f),h\rangle = \ln\delta\sum\limits_j\sum\limits_{\tau\in I_j}|Q_{\tau}^j|\ \Big\langle q_j(\cdot,y_{Q_{\tau}^j}), h(\cdot)\Big\rangle\ q_{\delta^{-2j}}(L)f(y_{Q_{\tau}^j}).
\end{eqnarray*}

We first claim that
$$ \Big\langle q_j(\cdot,y_{Q_{\tau}^j}), h(\cdot)\Big\rangle =\overline{\bar{q}(\delta^{-2j}L^*)h(y_{Q_\tau^j})}.
$$

In fact, for every $h\in L^{p'}$,  one can write
\begin{align}\label{s-bar of q}
&\Big\langle q_j(\cdot, y_{Q_\tau^j}), h(\cdot)\Big\rangle=\int_{\RN} K_{q(\delta^{-2j}L)}(x, y_{Q_\tau^j}) \overline{h(x)} \, dx\nonumber\\
&=\int_{\RN} \overline{ K_{\big(q(\delta^{-2j}L)\big)^*}(y_{Q_\tau^j},x)}\ \overline{h(x)} \, dx
= \overline{\int_{\RN}K_{\big(\bar{q}(\delta^{-2j}L^*)\big)}(y_{Q_\tau^j},x) {h(x)} \, dx}\nonumber\\
&=\overline{\bar{q}(\delta^{-2j}L^*)h(y_{Q_\tau^j})},
\end{align}
where we use $K_{\psi(tL)}(x,y)$ to denote the kernel of the operator $\psi(tL)$.

Then we combine  (\ref{s-bar of q}), Lemma
\ref{le2.1}, Remark \ref{remark g function Lp} and the H\"older inequality to obtain
\begin{eqnarray*}
\big|\langle T_ \delta(f),h\rangle\big| &\leq& \ln\delta
\sum\limits_j\sum\limits_{\tau\in I_j}|Q_{\tau}^j|\big|q_{\delta^{-2j}}(L)f(y_{Q_{\tau}^j})\big|\,\big|\bar{q}(\delta^{-2j}L^*){h}(y_{Q_{\tau}^j})\big|\\
&\leq&\ln\delta\int_{\RN} \sum\limits_j\sum\limits_{\tau\in
I_j}\big|q_{\delta^{-2j}}(L)f(y_{Q_{\tau}^j})\big|\,\big|\bar{q}(\delta^{-2j}L^*){h}(y_{Q_{\tau}^j})\big|
\chi_{Q_{\tau}^j}(x)\,dx\\
&\leq& \big\|g_{1, \delta}(L)f\big\|_p
\big\| g^*_{1, \delta}(L^*){h}  \big\|_{p'}\\
&\leq& C\|f\|_p\|h\|_{p'},
\end{eqnarray*}
which, together with the fact that $L^p(\RN)\cap L^2(\RN)$ is dense in $L^p(\RN)$ and $L^{p'}(\RN)\cap L^2(\RN)$ is dense in $L^{p'}(\RN)$, implies that
\begin{eqnarray*}
\|T_ \delta(f)\|_p=\sup\limits_{\|h\|_{p'}\leq 1}\big|\langle
T_ \delta(f),h\rangle\big| \leq C
\|f\|_p.
\end{eqnarray*}
This finishes the proof of Theorem \ref{th2.1}.
\end{proof}

We now introduce the remainder operator $R_\delta$.
\begin{definition}
Let $T_\delta$ be the same as in \eqref{Tdelta}. We now set
$$R_ \delta:=I-T_ \delta,$$
where $I$ is the identity operator on $L^2(\RN)$.
\end{definition}

\begin{theorem}\label{th2.2}
 Then there exists a
constant $1<\delta<2$ such that  $\|R_\delta(f)\|_p\leq \frac{1}{2}\|f\|_p$,
for all $1<p<\infty$.
\end{theorem}

\begin{proof}
For any $f\in L^2(\RN)\cap L^p(\RN)$, one can write (by using $H_\infty$-functional calculus \cite{Mc})
$$
f(x)=\int_0^\infty\!\! \int_{\RN} q_{t^2}(x,y)
q_{t^2}(L)f(y)\,dy\frac{dt}{t}=\sum\limits_j\sum\limits_{\tau\in
I_j}\int_{ \delta^{-j}}^{ \delta^{-j+1}}
\!\!\int_{Q_{\tau}^j}q_{t^2}(x,y)q_{t^2}(L)f(y)dy\frac{dt}{t},
$$
where the last equality follows from the argument as in \eqref{se3.1 final}
in the sense of $L^2(\RN)$.

\noindent We then decompose $R_ \delta(f)$ into four terms:
$R_ \delta(f)=\sum\limits_{i=1}^4R_{\delta,i}(f)$, where for every $f\in L^2(\RN)$,

$$
R_{\delta,1}(f):=\sum\limits_j\sum\limits_{\tau\in
I_j}\int_{ \delta^{-j}}^{ \delta^{-j+1}}
\int_{Q_{\tau}^j}[q_{t^2}(x,y)-q_{t^2}(x,y_{Q_{\tau}^j})]q_{t^2}(L)f(y)dy\frac{dt}{t};
$$
$$
R_{\delta,2}(f):=\sum\limits_j\sum\limits_{\tau\in
I_j}\int_{ \delta^{-j}}^{ \delta^{-j+1}}
\int_{Q_{\tau}^j}[q_{t^2}(x,y_{Q_{\tau}^j})-q_j(x,y_{Q_{\tau}^j})]q_{t^2}(L)f(y)dy\frac{dt}{t};
$$
$$
R_{\delta,3}(f):=\sum\limits_j\sum\limits_{\tau\in
I_j}\int_{ \delta^{-j}}^{ \delta^{-j+1}}
\int_{Q_{\tau}^j}q_j(x,y_{Q_{\tau}^j})[q_{t^2}(L)f(y)-q_{t^2}(L)f(y_{Q_{\tau}^j})]dy\frac{dt}{t};
$$
$$
R_{\delta,4}(f):=\sum\limits_j\sum\limits_{\tau\in
I_j}\int_{ \delta^{-j}}^{ \delta^{-j+1}}
\int_{Q_{\tau}^j}q_j(x,y_{Q_{\tau}^j})[q_{t^2}(L)f(y_{Q_{\tau}^j})-q_{\delta^{-2j}}(L)f(y_{Q_{\tau}^j})]dy\frac{dt}{t}.
$$

We point out that, by repeating the argument as in the proof of Theorem \ref{Thm Tdelta}, we obtain that
all the above operators $R_{\delta,i}$, $i=1,2,3,4$, are well-defined and the series converges
in the sense of $L^2(\RN)$.

We now first estimate the norm of $R_{\delta,1}(f)$.  Applying (\ref{s-bar of q}), we have that for every $f\in L^p(\RN)\cap L^2(\RN)$,
\begin{align*}
&\|R_{\delta,1}(f)\|_p\\
 &=\sup_{ g\in L^{p'}(\RN)\cap L^2(\RN):\ \|g\|_{L^{p'}(\RN)}\leq1 } |\langle R_{\delta,1}(f),g \rangle|\\
&= \sup_{ g\in L^{p'}(\RN)\cap L^2(\RN):\ \|g\|_{L^{p'}(\RN)}\leq1 }
\bigg| \sum\limits_j\sum\limits_{\tau\in
I_j}\bigg\langle\int_{ \delta^{-j}}^{ \delta^{-j+1}}
\int_{Q_{\tau}^j}[q_{t^2}(x,y)-q_{t^2}(x,y_{Q_{\tau}^j})]q_{t^2}f(y)dy\frac{dt}{t} ,g \bigg\rangle \bigg|
 \\
&\leq \sup\limits_{ \|g\|_{L^{p'}(\RN)}\leq1 }\sum\limits_j\sum\limits_{\tau\in
I_j}\int_{ \delta^{-j}}^{ \delta^{-j+1}}
\int_{Q_{\tau}^j}|\bar{q}{(t^2L^*)}{g}(y)-\bar{q}{(t^2L^*)}g(y_{Q_{\tau}^j})||q_{t^2}(L)f(y)|dy\frac{dt}{t},
\end{align*}
where the second equality follows from the fact that  $R_{\delta,1}$
is well-defined, and the series converges in the sense of $L^2(\RN)$, and the inequality follows from
Fubini's theorem.

\noindent Note that $y\in Q_\tau^j,\ \delta^{-j}\leq t\leq \delta^{-j+1}$, we use (\ref{s-e2.2}) to get
\begin{eqnarray*}
|\bar{q}{(t^2L^*)}{g}(y)-\bar{q}{(t^2L^*)}{g}(y_{Q_{\tau}^j})|&\leq& C\bigg(\frac{|y-y_{Q_{\tau}^j}|}{t}\bigg) \inf\limits_{u\in
Q_{\tau}^j}\mathcal{M}\big(\bar{\varphi}(t^2L^*)g\big)(u)
\\
&\leq& C\delta^{(-M+1)}\inf\limits_{u\in
Q_{\tau}^j}\mathcal{M}\big(\bar{\varphi}(t^2L^*)g\big)(u).
\end{eqnarray*}

\noindent By H\"older's inequality,  vector-value maximal
theorem and Lemma \ref{g fucntion Lp},
we obtain
\begin{align}\label{se2.1}
\|R_{\delta,1}(f)\|_p&\leq C  \delta^{(-M+1)} \sup\limits_{\|g\|_{p'}\leq
1}\sum\limits_j\sum\limits_{\tau\in
I_j}\int_{{ \delta^{-j}}}^{ \delta^{-j+1}} \int_{Q_{\tau}^j}
\inf\limits_{u\in
Q_{\tau}^j}\mathcal{M}\big(\bar{\varphi}(t^2L^*)g\big)(u)\big|q_{t^2}(L)f(y)\big|dy\frac{dt}{t}\nonumber\\
&\leq C  \delta^{(-M+1)}\sup\limits_{\|g\|_{p'}\leq
1}\int_0^\infty\int_{\RN}\mathcal{M}\big(\bar{\varphi}(t^2L^*)g\big)(y)\big|q_{t^2}(L)f(y)\big|dy\frac{dt}{t}\nonumber\\
&\leq C  \delta^{(-M+1)}\sup\limits_{\|g\|_{p'}\leq
1}\Big\|\Big(\int_0^\infty\big|q_{t^2}(L)f(y)\big|^2\frac{dt}{t}\Big)^{1/2}\Big\|_p
\Big\|\Big(\int_0^\infty\mathcal{M}\big(\bar{\varphi}(t^2L^*)g\big)(y)^2\frac{dt}{t}\Big)^{1/2}\Big\|_{p'}\\
&\leq C  \delta^{(-M+1)}\|f\|_p.\nonumber
\end{align}

By similar argument, we have
\begin{eqnarray*}
\|R_{\delta,3}(f)\|_p&\leq& C  \delta^{(-M+1)} \sup\limits_{\|g\|_{p'}\leq
1}\sum\limits_j\sum\limits_{\tau\in
I_j}\int_{{ \delta^{-j}}}^{ \delta^{-j+1}} \int_{Q_{\tau}^j}
\inf\limits_{u\in
Q_{\tau}^j}\mathcal{M}\big(\varphi_{t^2}(L)(f)\big)(u)|\bar{q}(\delta^{-2j}L^*)g(y_{Q_{\tau}^j})|dy\frac{dt}{t}\\
&=& C  \delta^{(-M+1)} \sup\limits_{\|g\|_{p'}\leq
1}\sum\limits_j\sum\limits_{\tau\in
I_j}|Q_{\tau}^j|\int_{{ \delta^{-j}}}^{ \delta^{-j+1}}
 \inf\limits_{u\in
Q_{\tau}^j}\mathcal{M}\big(\varphi_{t^2}(L)(f)\big)(u)|\bar{q}(\delta^{-2j}L^*)g(y_{Q_{\tau}^j})|\frac{dt}{t}
\end{eqnarray*}

\noindent To continue, one can write
\begin{align*}
\sum\limits_j\sum\limits_{\tau\in I_j}  & |Q_{\tau}^j|
\int_{{ \delta^{-j}}}^{ \delta^{-j+1}}\inf\limits_{y\in
Q_{\tau}^j}\mathcal{M}\big(\varphi_{t^2}(L)(f)\big)(y)|\bar{q}(\delta^{-2j}L^*){g}(y_{Q_{\tau}^j})|\frac{dt}{t}\\
&\leq C\int_{\RN} \sum\limits_j\sum\limits_{\tau\in I_j}
\int_{{ \delta^{-j}}}^{ \delta^{-j+1}}\inf\limits_{y\in
Q_{\tau}^j}\mathcal{M}\big(\varphi_{t^2}(L)(f)\big)(y)|\bar{q}(\delta^{-2j}L^*){g}(y_{Q_{\tau}^j})|\chi_{Q_{\tau}^j}(x)\frac{dt}{t}dx\\
&\leq C \int_{\RN}
\sum\limits_j\sum\limits_{\tau\in I_j} \int_{{ \delta^{-j}}}^{ \delta^{-j+1}} \mathcal{M}\big(\varphi_{t^2}(L)(f)\big)(x)|\bar{q}(\delta^{-2j}L^*){g}(y_{Q_{\tau}^j})|\chi_{Q_{\tau}^j}(x)\frac{dt}{t}dx\\
&\leq C \Big\|\Big(\sum\limits_j\sum\limits_{\tau\in
I_j}\int_{{ \delta^{-j}}}^{ \delta^{-j+1}}|\mathcal{M}\big(\varphi_{t^2}(L)(f)\big)(x)|^2\frac{dt}{t}\chi_{Q_{\tau}^j}(x)\Big)^{1/2}
\Big\|_p\\
& \hskip 3cm \times \Big\|\Big(\sum\limits_j\sum\limits_{\tau\in
I_j}\int_{{ \delta^{-j}}}^{ \delta^{-j+1}}|\bar{q}(\delta^{-2j}L^*){g}(y_{Q_{\tau}^j})|^2\frac{dt}{t}\chi_{Q_{\tau}^j}(x)\Big)^{1/2}
\Big\|_{p'}\\
&\leq C
\Big\|\Big(\int_0^{\infty}\big(\mathcal{M}\big(\varphi_{t^2}(L)(f)\big)\big)^2\frac{dt}{t}\Big)^{1/2}
\Big\|_p\big\|
g^*_{1, \delta}(L^*){g}\big\|_{p'}\\
&\leq  C \big\|f\big\|_p
\big\|g\big\|_{p'},
\end{align*}
where in the last inequality above we have used Remark \ref{remark g function Lp}.   Therefore,  we show that
\begin{eqnarray}\label{se2.2}
\|R_{\delta,3}(f)\|_p&\leq& C \delta^{(-M+1)} \|f\|_p.
\end{eqnarray}

\bigskip

As for  $R_{\delta,4}(f)$,  we apply dual argument and \eqref{s-bar of q} to write
\begin{align}\label{e2.4}
&\|R_{\delta,4}(f)\|_p\nonumber\\
&\leq\sup\limits_{\|g\|_{p'}\leq 1}
\sum\limits_j\sum\limits_{\tau\in
I_j}\int_{ \delta^{-j}}^{ \delta^{-j+1}}\int_{Q_{\tau}^j}
\big|\bar{q}(\delta^{-2j}L^*){g}(y_{Q_{\tau}^j})\big|\big|q_{t^2}(L)f(y_{Q_{\tau}^j})-q_{\delta^{-2j}}(L)f(y_{Q_{\tau}^j})\big|dy\frac{dt}{t}\nonumber\\
&=\sup\limits_{\|g\|_{p'}\leq 1} \int_{\RN}
\sum\limits_j\sum\limits_{\tau\in
I_j}\int_{ \delta^{-j}}^{ \delta^{-j+1}}
\big|\bar{q}(\delta^{-2j}L^*){g}(y_{Q_{\tau}^j})\big|\big|q_{t^2}(L)f(y_{Q_{\tau}^j})-q_{\delta^{-2j}}(L)f(y_{Q_{\tau}^j})\big|\chi_{Q_{\tau}^j}(x)\frac{dt}{t}dx\nonumber\\
&\leq \sup\limits_{\|g\|_{p'}\leq 1} \sqrt{\ln \delta}\Big\|\Big(
\sum\limits_j\sum\limits_{\tau\in I_j}
\big|\bar{q}(\delta^{-2j}L^*){g}(y_{Q_{\tau}^j})\big|^2\chi_{Q_{\tau}^j}\Big)^{1/2}\Big\|_{p'}\nonumber\\
&\qquad \times\Big\|\Big(\sum\limits_j\sum\limits_{\tau\in
I_j}\int_{ \delta^{-j}}^{ \delta^{-j+1}}\big|q_{t^2}(L)f(y_{Q_{\tau}^j})-q_{\delta^{-2j}}(L)f(y_{Q_{\tau}^j})\big|^2
\frac{dt}{t}\chi_{Q_{\tau}^j}\Big)^{1/2}\Big\|_p\nonumber\\
&\leq C\Big\|\Big(\sum\limits_j\sum\limits_{\tau\in
I_j}\int_{ \delta^{-j}}^{ \delta^{-j+1}}\big|q_{t^2}(L)f(y_{Q_{\tau}^j})-q_{\delta^{-2j}}(L)f(y_{Q_{\tau}^j})\big|^2
\frac{dt}{t}\chi_{Q_{\tau}^j}\Big)^{1/2}\Big\|_p,
\end{align}

\noindent where in the last inequality we have used Remark \ref{remark g function Lp}.

 For $ \delta^{-j}\leq t< \delta^{-j+1}$, we use (\ref{s-e2.3}) to get
\begin{align}\label{e2.5}
|q_{t^2}(L)  & f(y_{Q_{\tau}^j})-q_{\delta^{-2j}}(L)f(y_{Q_{\tau}^j})|
\leq C \sqrt{\ln \delta}
\Big(\int_{ \delta^{-j}}^{ \delta^{-j+1}}\big| s^2Lq^{(1)}(s^2L) f(y_{Q_{\tau}^j})\big|^2\frac{ds}{s}\Big)^{1/2}.
\end{align}

\noindent Substituting (\ref{e2.5}) into (\ref{e2.4}) and applying
Lemma \ref{le2.1},  we have
\begin{align}\label{s-e3.1}
\Big\|\Big(\sum\limits_j&\sum\limits_{\tau\in
I_j}\int_{ \delta^{-j}}^{ \delta^{-j+1}}\big|q_{t^2}(L)f(y_{Q_{\tau}^j})-q_{\delta^{-2j}}(L)f(y_{Q_{\tau}^j})\big|^2
\frac{dt}{t}\chi_{Q_{\tau}^j}\Big)^{1/2}\Big\|_p\nonumber\\
&\leq C\sqrt{\ln \delta}\Big\|\Big(\sum\limits_j\sum\limits_{\tau\in I_j}
\int_{ \delta^{-j}}^{ \delta^{-j+1}}| s^2Lq^{(1)}(s^2L)  f(y_{Q_{\tau}^j})|^2\frac{ds}{s}
\chi_{Q_{\tau}^j}\Big)^{1/2}\Big\|_p\nonumber\\
&= C\sqrt{\ln \delta} \big\|g_{4, \delta}(L)f\big\|_p\nonumber\\
&\leq C\sqrt{\ln \delta}\|f\|_p.
\end{align}

\noindent Observe  that if $1< \delta <2$, then
$\ln  \delta\leq ( \delta-1)$. Thus, we have
\begin{eqnarray}\label{se2.3}
\|R_{\delta,4}(f)\|_p \leq C
\sqrt{(\delta -1)}  \|f\|_p.
\end{eqnarray}

For the term $R_{\delta,2}(f)$. We note that for every $y\in Q_{\tau}^j$
and $ \delta^{-j}\leq t< \delta^{-j+1}$,
$$|q_{t^2}(L)f(y)|=2 \Big|\int_{\RN} K_{\zeta(t^2L)}(y,z)\varphi(t^2L)f(z)dz\Big|\leq C\inf\limits_{u\in Q_{\tau}^j}\mathcal{M}(\varphi_{t^2}f)(u).
$$

\noindent By dual argument and \eqref{s-bar of q}, we have
\begin{align*}
&\|R_{\delta,2}(f)\|_p\nonumber\\
&\leq\sup\limits_{\|g\|_{p'}\leq 1}
\sum\limits_j\sum\limits_{\tau\in
I_j}\int_{ \delta^{-j}}^{ \delta^{-j+1}}\int_{Q_{\tau}^j}
\big|q_{t^2}(L)f(y)\big|\big|\bar{q}(t^2L^*){g}(y_{Q_{\tau}^j})-\bar{q}(\delta^{-2j}L^*)({g})(y_{Q_{\tau}^j})\big|dy\frac{dt}{t}\nonumber\\
&\leq C\sup\limits_{\|g\|_{p'}\leq 1}
\sum\limits_j\sum\limits_{\tau\in
I_j}\int_{ \delta^{-j}}^{ \delta^{-j+1}}\int_{Q_{\tau}^j}
\inf\limits_{u\in
Q_{\tau}^j}\mathcal{M}(\varphi_{t^2}f)(u)\big|\bar{q}(t^2L^*){g}(y_{Q_{\tau}^j})-\bar{q}(\delta^{-2j}L^*)({g})(y_{Q_{\tau}^j})\big|dy\frac{dt}{t}\nonumber\\
&=C\sup\limits_{\|g\|_{p'}\leq 1}
\sum\limits_j\sum\limits_{\tau\in
I_j}\int_{ \delta^{-j}}^{ \delta^{-j+1}}|Q_{\tau}^j|
\inf\limits_{u\in
Q_{\tau}^j}\mathcal{M}(\varphi_{t^2}f)(u)\big|\bar{q}(t^2L^*){g}(y_{Q_{\tau}^j})-\bar{q}(\delta^{-2j}L^*)({g})(y_{Q_{\tau}^j})\big|\frac{dt}{t}\nonumber\\
&\leq C\sup\limits_{\|g\|_{p'}\leq 1}
\int\sum\limits_j\sum\limits_{\tau\in
I_j}\int_{ \delta^{-j}}^{ \delta^{-j+1}}
\mathcal{M}(\varphi_{t^2}f)(x)\big|\bar{q}(t^2L^*){g}(y_{Q_{\tau}^j})-\bar{q}(\delta^{-2j}L^*)({g})(y_{Q_{\tau}^j})\big|\frac{dt}{t}\chi_{Q_{\tau}^j}(x)dx\nonumber\\
&\leq C \sup\limits_{\|g\|_{p'}\leq 1} \Big\|\Big(
\int_0^{\infty}\big(\mathcal{M}(\varphi_{t^2}f)\big)^2\frac{dt}{t}
\Big)^{1/2}\Big\|_{p}\nonumber\\
&\qquad \times\Big\|\Big(\sum\limits_j\sum\limits_{\tau\in
I_j}\int_{ \delta^{-j}}^{ \delta^{-j+1}}\big|\bar{q}(t^2L^*){g}(y_{Q_{\tau}^j})-\bar{q}(\delta^{-2j}L^*)({g})(y_{Q_{\tau}^j})\big|^2
\frac{dt}{t}\chi_{Q_{\tau}^j}\Big)^{1/2}\Big\|_{p'}
\end{align*}
\noindent By \eqref{s-e3.1}, Remark \ref{remark g function Lp} and vector-value maximal theorem, we  have
\begin{eqnarray}\label{se2.4}
\|R_{\delta,2}(f)\|_p\leq C\sqrt{(\delta-1)}\|f\|_p.
\end{eqnarray}

Combining (\ref{se2.1}), (\ref{se2.2}), (\ref{se2.3}) and (\ref{se2.4}),  we have, there exists a constant $C_1>0$ such that
$$
\|R_ \delta(f)\|_p\leq C_1 \big( \delta^{-M}+\sqrt{(\delta-1)}\big)\|f\|_p.
$$

\noindent We can choose $\delta$ close to 1 and $M$ large enough, such
that $C_1 \big(\delta^{-M}+\sqrt{(\delta-1)}\big)<1/2$, which completes the proof.
\end{proof}

We now apply
 the estimates of Theorems \ref{th2.1} and \ref {th2.2} to prove Theorem \ref{main1}.

\begin{proof}[Proof of Theorem \ref{main1}]
By Theorem \ref{th2.2}, we have that for each $k=1,2,\cdots$, \ \
$\|\big(R_\delta\big)^k(f)\|_p\leq 2^{-k}\|f\|_p$.  Therefore,  the operator $T_\delta(f)$
is invertible, and
\begin{eqnarray}
\big\|T_\delta^{-1}f\big\|_p=\big\|\big(I-R_\delta\big)^{-1} f\big\|_p\leq
\sum\limits_{k=0}^{+\infty}\big\|\big(R_\delta\big)^k(f)\big\|_p\leq
2\big\|f\big\|_p.
\end{eqnarray}

For every $f\in L^p(\mathbb{R}^n)$, one can write
\begin{eqnarray}\label{ee2.6}
f=T_\delta\circ T_\delta^{-1}f=\sum\limits_j\sum\limits_{\tau\in
I_j}\langle T_\delta^{-1}f,
 \psi^*_{j,\tau}\rangle  \psi_{j,\tau}  \hskip 2cm  in \ \ L^p(\mathbb{R}^n).
\end{eqnarray}

\noindent Applying Lemma \ref{le2.1}, we have
\begin{align*}
 \bigg\|\bigg( \sum\limits_j
\sum\limits_{\tau\in I_j}\big|\langle T_\delta^{-1}f,
\psi^*_{j,\tau}\rangle\big|^2
\frac{1}{|Q_{\tau}^j|}\chi_{Q_{\tau}^j}\bigg)^{1\over 2}\bigg\|_p =\Big\|g_{1,\delta}(L)(T_\delta^{-1}f)\Big\|_p
\leq C\big\|T_\delta^{-1}f\big\|_p\leq C\big\|f\big\|_p.
\end{align*}

For the left inequality, we  use the dual argument to write
\begin{eqnarray*}
\big\|f\big\|_p=\sup\limits_{\|g\|_{p'}\leq 1}\big|\langle f,
g\rangle\big|,
\end{eqnarray*}
where $p'$ is the adjoint number of $p$.

For each $f\in L^p(\mathbb{R}^n)$ and $g\in L^{p'}(\mathbb{R}^n)$,
using the equality (\ref{ee2.6}) we have
\begin{align*}
\big|\langle f, g\rangle\big|&=\ln
 \delta\Big|\sum\limits_j\sum\limits_{\tau\in I_j}\langle T_\delta^{-1}f,
\psi^*_{j,\tau}\rangle \langle \psi_{j,\tau}, g\rangle\Big|\\
&=\ln  \delta\Big|\int_{\RN} \sum\limits_j\sum\limits_{\tau\in I_j}\langle
T_\delta^{-1}f,
\psi^*_{j,\tau}\rangle \langle\psi_{j,\tau}, g\rangle\frac{1}{|Q^j_{\tau}|}\chi_{Q^j_{\tau}}(x) dx\Big|.
\end{align*}
Applying H\"older inequality and Lemma \ref{le2.1}, one
writes
\begin{align*}
\Big|\langle f, g\rangle\Big|
 &\leq\Big\|\Big(
\sum\limits_j\sum\limits_{\tau\in I_j}\big|\langle T_\delta^{-1}f,
\psi^*_{j,\tau}\rangle\big|^2\frac{1}{|Q^j_{\tau}|}\chi_{Q^j_{\tau}}\Big)^{1\over
2}\Big\|_p\cdot \Big\|\Big( \sum\limits_j\sum\limits_{\tau\in
I_j}\big|\langle g, \psi_{j,\tau}\rangle\big|^2
\frac{1}{|Q^j_{\tau}|}\chi_{Q^j_{\tau}}\Big)^{1\over 2}\Big\|_{p'}\\
&\leq C\Big\|\Big( \sum\limits_j\sum\limits_{\tau\in
I_j}\big|\langle T_\delta^{-1}f,
\psi^*_{j,\tau}\rangle\big|^2\frac{1}{|Q^j_{\tau}|}\chi_{Q^j_{\tau}}\Big)^{1\over
2}\Big\|_p\big\|g\big\|_{p'}.
\end{align*}

Therefore, we obtain that
$$
\big\|f\big\|_p\leq C\bigg\|\bigg( \sum\limits_j\sum\limits_{\tau\in
I_j}\big|\langle T_\delta^{-1}f,
\psi^*_{j,\tau}\rangle\big|^2\frac{1}{|Q^j_{\tau}|}\chi_{Q^j_{\tau}}\bigg)^{1\over
2}\bigg\|_p,
$$

\noindent which completes the proof.
\end{proof}

\vspace{0.3cm}
\section{Frame decompositions on $H_L^1(\mathbb{R}^n)$ }
 \setcounter{equation}{0}

We first recall the tent space $T_2^1(\mathbb{R}^{n+1})$ and the molecules for the Hardy space $H_L^1(\mathbb{R}^n)$.

In \cite{CMS}, Coifman, Meyer and Stein introduced and studied a new family of function spaces, the so-called ``tent spaces''.
For any function $f(y,t)$ defined on $\mathbb{R}^{n+1}$ we
will denote
$$ \mathcal{A}(f)(x)=\Big( \int_{0}^\infty \int_{|x-y|<t} |f(y,t)|^2 {dydt\over t^{n+1}} \Big)^{1/2}. $$
As in \cite{CMS}, the tent space $T_2^1$ is defined as the space of functions f such that  $\mathcal{A}(f) \in L^1 (\mathbb{R}^n)$. The resulting equivalence classes are then equipped with the norm $\|f\|_{T_2^1} = \|\mathcal{A}(f)\|_1$.

Next, a function $a(x,t)$ is called a $T_2^1$-atom if

(i) the function $a(x,t)$ is supported in $\hat{B}$ (for some ball $B\subset \mathbb{R}^n$);

(ii) $\int_{\hat{B}}|a(x,t)|^2 {dxdt\over t} \leq |B|^{-1}$,

\noindent where $\hat{B}$ is the tent of the ball $B$, defined as $\hat{B}=\{ (y,t)\in\mathbb{R}^n\times \mathbb{R}_+:\ B(y,t)\subset B \}$, and
$B(y,t)$ is the ball in $\mathbb{R}^n$ centered at $y$ with radius $t$.

Recall that in \cite{AuDM}, a function $m(x)$ is called an $L$-molecule if
\begin{align}\label{def molecule}
m(x)=\int_0^\infty t^2Le^{-t^2L}( a(\cdot,t))(x) {dt\over t},
\end{align}
where $a(t, x)$ is a $T_2^1$-atom as defined above. An $L$-molecule decomposition of $f$ in the space $H_L^1$ is first obtained in Theorem 7 of \cite{AuDM}. Here we refer to the following statement as in \cite{DY2}.

\begin{prop}[\cite{DY2}]
Let $f\in H_L^1(\mathbb{R}^n)\cap L^2(\mathbb{R}^n)$. There exist $L$-molecules $m_k$ and numbers $\lambda_k$ for $k=0,1,2,\ldots$, such that
\begin{align}\label{molecule}
f(x)=\sum_k\lambda_km_k(x).
\end{align}
The sequence $\{\lambda_k\}$ satisfies $\sum_k|\lambda_k|\leq C\|f\|_{H_L^1(\mathbb{R}^n)}$. Conversely, the decomposition \eqref{molecule} satisfies
$$ \|f\|_{H_L^1(\mathbb{R}^n)} \leq  C \sum_k|\lambda_k|.$$
\end{prop}
We point out that the equality \eqref{molecule} holds in the sense of $L^2(\RN)$, for
more detail of explanation, we refer to Proposition 3.23 in \cite{DL}.

\subsection{Littlewood--Paley $g$ functions on $H_L^1(\mathbb{R}^n)$}

Next we prove that the four auxiliary Littlewood--Paley $g$ functions as defined in Section 3.1 are bounded from $H_L^1(\mathbb{R}^n)$ to $L^1(\mathbb{R}^n)$.

\begin{lemma}\label{lemma H1 to L1 molecule}Assume that $L$ satisfies ${\bf(H 1)}$, ${\bf(H 2)}$ and ${\bf(H 3)}$.
There exists a positive constant $C$, such that for  any fixed $1<\delta\leq2$  and every $f\in
H_L^1(\mathbb{R}^n)$,
\begin{eqnarray}\label{e1 lemma H1 to L1 molecule}
\big\|g_{i,\delta}(L)(f)\big\|_{L^1(\mathbb{R}^n)}\leq
C\big\|f\big\|_{H^1_L(\mathbb{R}^n)},
\end{eqnarray}
where $i=1,2,3,4$.
\end{lemma}

\begin{proof}
We now verify (\ref{e1 lemma H1 to L1 molecule}) for $i=1$.

Note that $g_{1,\delta}(L)$ is non-negative, sublinear, and bounded on $L^2(\RN)$ and note also that
for every $f\in H^1_L(\mathbb{R}^n)\cap L^2(\RN)$, $f$ has the following molecular decomposition
$$f = \sum_{k=1}^\infty\lambda_k m_k$$
in the sense of $L^2(\RN)$ with $\sum_k|\lambda_k| \approx \|f\|_{H^1_L(\mathbb{R}^n)}$.

Note that $\sum_{k=1}^\infty\lambda_km_k = f$ in the sense of $L^2(\RN)$. So we have that
$$\lim_{N\to\infty } F_N =0\quad {\rm in}\quad L^2(\RN),$$
where $F_N :=\sum_{k=N+1}^\infty\lambda_km_k $. So we have
$$\lim_{N\to\infty } g_{1,\delta}(L) \big(F_N\big)   =0\quad {\rm in}\quad L^2(\RN).$$

So there exists a subsequence $\{g_{1,\delta}(L) \big(F_{N_j}\big)\}$
such that
$$ \lim_{j\to\infty} g_{1,\delta}(L) \big(F_{N_j}\big) =0\quad {\rm a.e.\ in}\quad \RN.  $$
Then for almost every $x\in\RN$, for any $\varepsilon>0$, there exists $J>0$ sufficiently large such that for every integer $j>J$, we have
\begin{align*}
|g_{1,\delta}(L)f(x)| &= \bigg|g_{1,\delta}(L) \bigg( \sum_{k=1}^\infty\lambda_km_k\bigg)(x)\bigg| = \bigg|g_{1,\delta}(L) \bigg(\sum_{k=1}^{N_j}\lambda_km_k + \sum_{k=N_j+1}^\infty\lambda_km_k\bigg)(x)\bigg| \\
&\leq \bigg|g_{1,\delta}(L) \bigg(\sum_{k=1}^{N_j}\lambda_km_k\bigg)(x)\bigg| + \bigg|g_{1,\delta}(L)\bigg( \sum_{k=N_j+1}^\infty\lambda_km_k\bigg)(x)\bigg|\\
&\leq \sum_{k=1}^{N_j}|\lambda_k|\, \big|g_{1,\delta}(L)(m_k)(x) \big|+\varepsilon,
\end{align*}
which gives
$$ |g_{1,\delta}(L)f(x)| \leq \sum_{k=1}^{\infty}|\lambda_k|\, \big|g_{1,\delta}(L)(m_k)(x) \big|, \quad {\rm a.e.}\ x\in\RN.$$
Hence we obtain that
$$ \|g_{1,\delta}(L)f\|_{L^1(\RN)} \leq \sum_{k=1}^{\infty}|\lambda_k|\, \big\|g_{1,\delta}(L)(m_k) \big\|_{L^1(\RN)}.$$

As a consequence, to prove \eqref{e1 lemma H1 to L1 molecule}, it suffices to prove that  there exists a positive constant $C$ independent of $\delta$ such that
  for every molecule $m$ as defined in \eqref{def molecule}, the following estimate
\begin{eqnarray}\label{g1 uni bd on molecule m}
\|g_{1,\delta}(L)(m)\|_{L^1(\mathbb{R}^n)}\leq C.
\end{eqnarray}
holds. To verify this,  we first see that
\begin{eqnarray*}
\|g_{1,\delta}(L)(m)\|_{L^1(\mathbb{R}^n)}=\int_{4B}|g_{1,\delta}(L)(m)(x)|dx+\int_{(4B)^c}|g_{1,\delta}(L)(m)(x)|dx=:I+II,
\end{eqnarray*}
where $B$ is the ball associated to $m$.

As for $I$, we have
\begin{align}\label{I}
I&=\int_{4B}|g_{1,\delta}(L)(m)(x)|dx \leq  |4B|^{1/2}\|g_{1,\delta}(L)(m)\|_{L^2(\mathbb{R}^n)}\nonumber\\
&\leq  C|4B|^{1/2}\|m\|_{L^2(\mathbb{R}^n)}\leq C |4B|^{1/2} \big(\int_0^{r_B }\!\!\int_{B}|a(x,t)|^2 dx  {dt\over t}\big)^{1/2}\leq C,
\end{align}
where in the second inequality we have used Lemma \ref{le2.1} and in the third inequality we have used Lemma 4.3 in \cite{DY2}  and the definition of  the molecule $m$.

Now we turn to $II$.
\begin{align}\label{so-e3.1}
II&=\sqrt{\ln \delta}\int_{(4B)^c}\Big(\sum_j\sum_{\tau\in
I_j}|q_{j}(m)(y_{Q_{\tau}^j})|^2\chi_{Q_{\tau}^j}(x)\Big)^{1/2}dx\\
&= \sqrt{\ln \delta}\int_{(4B)^c}\Big(\sum_{j}\sum_{\tau\in I_j}\Big| q_{j} \Big( \int_0^\infty t^2Le^{-t^2L}\big( a(\cdot,t)  \big) {dt\over t} \Big)(y_{Q_{\tau}^j}) \Big|^2\chi_{Q_{\tau}^j}(x)\Big)^{1/2}dx\nonumber\\
&=\sqrt{\ln \delta} \int_{(4B)^c}\Big(\sum_{j}\sum_{\tau\in I_j}\Big|   \int_0^{r_B } \int_B K_{q_{j}t^2Le^{-t^2L}}(y_{Q_{\tau}^j},y)  a(y,t)   {dydt\over t}  \Big|^2\chi_{Q_{\tau}^j}(x)\Big)^{1/2}dx\nonumber\\
&\leq\sqrt{\ln \delta}\int_{(4B)^c}\Big(\sum_{j}\sum_{\tau\in
I_j}\Big| \int_0^{r_B }\int_{B}  \Big({\delta^{-j}\over t}\wedge {t\over \delta^{-j}}\Big) {(\delta^{-j}+ t)\over (\delta^{-j}+ t+|y_{Q_{\tau}^j}-y|)^{n+1}}\nonumber\\
&\hskip 7cm |a(y,t)|{dydt\over t}\Big|^2\chi_{Q_{\tau}^j}(x)\Big)^{1/2}dx, \nonumber
\end{align}
where in the last inequality we have used the similar argument as that of (\ref{s-e2.1}).
Denote by $y_B$  the center of the ball $B$. Observe that for $0<t<r_B, x\in Q_{\tau}^j\cap (4B)^c, y\in B$,  there holds
\begin{align}\label{so-e3.2}
\delta^{-j}+ t+|y_{Q_{\tau}^j}-y|\geq C \big(\delta^{-j}+ t+|y_B-x|\big).
\end{align}
In fact, if $\delta^{-j} >r_B$, then
$$|y_B-x|\leq |x-y_{Q_{\tau}^j}|+|y_{Q_{\tau}^j}-y|+|y-y_B|\leq 2\delta^{-j}+ |y_{Q_{\tau}^j}-y|;$$
if $\delta^{-j}\leq r_B$, then $|y-y_B|\leq \frac{1}{2}|y_B-x|$, which implies that
$$|y_B-x|\leq |x-y_{Q_{\tau}^j}|+|y_{Q_{\tau}^j}-y|+\frac{1}{2}|y_B-x|,$$
and thus
$$|y_B-x|\leq 2(|x-y_{Q_{\tau}^j}|+|y_{Q_{\tau}^j}-y|)\leq 2\big(\delta^{-j}+ |y_{Q_{\tau}^j}-y|).$$

\smallskip

We insert the inequality (\ref{so-e3.2}) into (\ref{so-e3.1}), and get
\begin{align*}
II&\leq C\sqrt{\ln \delta}\int_{(4B)^c}\Big(\sum_{j}\Big| \int_0^{r_B }\!\!\int_{B}  \Big({\delta^{-j}\over t}\wedge {t\over \delta^{-j}}\Big) {(\delta^{-j}\vee t)^{1 \over 2}\over (\delta^{-j}\vee t+|x-y_B|)^{n+{1 \over 2}}}
|a(y,t)|{dydt\over t}\Big|^2\Big)^{1/2}dx\\
&\leq C\sqrt{\ln \delta}\int_{(4B)^c}\Big(\sum_{j}\Big| \int_0^{r_B }\!\!\int_{B}  \Big({\delta^{-j}\over t}\wedge {t\over \delta^{-j}}\Big) {(\delta^{-j}\vee t)^{1 \over 2}\over {|x-y_B|}^{n+{1 \over 2}}}
|a(y,t)|{dydt\over t}\Big|^2\Big)^{1/2}dx\\
&\leq C\sqrt{\ln \delta}\int_{(4B)^c}\Big(\sum_{j} \int_0^{r_B }\!\!\int_{B}  \Big({\delta^{-j}\over t}\wedge {t\over \delta^{-j}}\Big)^2 { (\delta^{-j}\vee t)\over {|x-y_B|}^{2n+1}}
{dydt\over t} \int_0^{r_B }\!\!\int_{B}
|a(y,t)|^2{dydt\over t}\Big)^{1/2}dx\\
&\leq C\sqrt{\ln \delta}\int_{(4B)^c}\Big(\sum_{j} \int_0^{r_B }  \Big({\delta^{-j}\over t}\wedge {t\over \delta^{-j}}\Big)^2 { (\delta^{-j}\vee t)\over {|x-y_B|}^{2n+1}}
{dt\over t}  \Big)^{1/2}dx\\
&\leq C\sqrt{\ln \delta}\int_{(4B)^c} {1\over |x-y_B|^{n+{1 \over 2}}}dx \, \Big(\sum_{j} \int_0^{r_B }  \Big({\delta^{-j}\over t}\wedge {t\over \delta^{-j}}\Big)^{2}  (\delta^{-j}\vee t)
{dt\over t}  \Big)^{1/2},
\end{align*}
where the fourth inequality follows from property (ii) of the $T_2^1$ atom $a$, and we use $s\vee t$ to denote $\max\{s,t\}$ for every positive numbers $s$ and $t$.

 One can compute
\begin{align*}
\sum_{j }\int_0^{r_B }\Big({\delta^{-j}\over t}\wedge {t\over \delta^{-j}}\Big)^2  (\delta^{-j}\vee t){dt\over t}
&=\Big(\sum_{j:\ \delta^{-j}\leq r_B}+ \sum_{j:\ \delta^{-j}> r_B}\Big)\int_0^{r_B }  \Big({\delta^{-j}\over t}\wedge {t\over \delta^{-j}}\Big)^2  (\delta^{-j}\vee t)
{dt\over t}\\
&=:II_1+II_2.
\end{align*}
For the term $II_1$, we have
\begin{align*}
II_1&= \sum_{j:\ \delta^{-j}\leq r_B} \Big(\int_0^{\delta^{-j} }+ \int_{\delta^{-j}}^{r_B}\Big) \Big({\delta^{-j}\over t}\wedge {t\over \delta^{-j}}\Big)^2  (\delta^{-j}\vee t){dt\over t} \\
&\leq \sum_{j:\ \delta^{-j}\leq r_B} \int_0^{\delta^{-j} } \Big( {t\over \delta^{-j}}\Big)^2  \delta^{-j}{dt\over t}+\sum_{j:\ \delta^{-j}\leq r_B} \int_{\delta^{-j}}^{r_B} \Big( {\delta^{-j}\over t}\Big)^2  t {dt\over t}\\
&\leq C \sum_{j: \delta^{-j}\leq r_B} \delta^{-j}\leq C \frac{r_B}{\delta-1}.
\end{align*}
For the term $II_2$, we have
\begin{align*}
II_2&\leq \sum_{j:\ \delta^{-j}> r_B} \int_0^{r_B }  \Big({t\over \delta^{-j}}\Big)^2  \delta^{-j}
{dt\over t}  \leq C  \big(r_B\big)^2\sum_{j: \delta^{-j}> r_B} \delta^{j}\leq C \frac{r_B}{\delta-1}.
\end{align*}

Thus, we obtain that for any $1<\delta\leq 2$,
\begin{align*}
II&\leq  C \sqrt{\ln \delta} \frac{1}{\sqrt{\delta-1}}\int_{(4B)^c}\frac{r_B^{1 \over 2}}{{|x-y_B|}^{n+{1 \over 2}}}dx\leq   \tilde{C},
\end{align*}
where $\tilde{C}$ is independent of $\delta$.
 Combining the estimate of $I$ and $II$, we can obtain that  (\ref{g1 uni
bd on molecule m}) holds.

The proofs of \eqref{e1 lemma H1 to L1 molecule} for $i=2,3$
 are similar to that for $i=1$. The proof of \eqref{e1 lemma H1 to L1 molecule} for $i=4$ is similar to that for $i=1$, but via the almost orthogonality estimate \eqref{s-e2.1.1} in Proposition \ref{s-prop3.3}. We omit the details here.
\end{proof}

\subsection{Proof of frame decomposition on Hardy spaces}

The next two results give the estimates on the norms of the operators $T_{\delta}$ and $I - T_{\delta}$ on the Hardy spaces associated
with the operator $L$.

\begin{theorem}\label{prop of S bd on H1}
  For every  ${n\over n+1}<r<1$,  there exists a positive constant $C(n,r)$,  such that for every $1<\delta<2$ and $f\in H_L^1(\mathbb{R}^n)$,
\begin{eqnarray}
\|T_\delta(f)\|_{H_L^1(\mathbb{R}^n)}\leq
C(n,r) \delta^{ Mn({1\over r}-1) }  \|f\|_{H_L^1(\mathbb{R}^n)}.
\end{eqnarray}
\end{theorem}

\begin{proof}
For any $f\in H_L^1(\mathbb{R}^n)\cap L^2(\mathbb{R}^n)$,
\begin{align*}
&\|T_\delta(f)\|_{H_L^1(\mathbb{R}^n)}\\
&=\Big\|\Big\{ \int_{\Gamma(x)} \big| q_{s^2}(L)\big(T_\delta(f)\big)(y)\big|^2
{dyds\over s^{n+1}}  \Big\}^{1\over
2}\Big\|_{L^1(\mathbb{R}^n)}\\
&=\ln\delta\Big\|\Big\{ \int_0^\infty \!\int_{|x-y|<s} \Big|
q_{s^2}(L)\Big(\sum\limits_j\sum_{\tau\in
I_j}|Q_{\tau}^j|q_{j}(\cdot,y_{Q_{\tau}^j})q_{\delta^{-2j}}(L)(f)(y_{Q_{\tau}^j})\Big)(y)\Big|^2 {dyds\over s^{n+1}}
\Big\}^{1\over 2}\Big\|_{L^1(\mathbb{R}^n)}\\
&=\ln\delta\Big\|\Big\{ \sum_k  \int_{\delta^{-k}}^{\delta^{-k+1}} \int_{|x-y|<s}\Big| \sum\limits_j\sum_{\tau\in
I_j}|Q_{\tau}^j|{q_{s^2}(L)q_{\delta^{-2j}}(L)}(y,y_{Q_{\tau}^j})q_{\delta^{-2j}}(L)(f)(y_{Q_{\tau}^j})\Big|^2 {dyds\over s^{n+1}}
\Big\}^{1\over 2}\Big\|_{L^1(\mathbb{R}^n)}\\
&\leq\ln\delta\Big\|\Big\{\sum_k  \int_{\delta^{-k}}^{\delta^{-k+1}} \int_{|x-y|<s}\\
&\quad\quad\quad\Big| \sum\limits_j\sum_{\tau\in
I_j}|Q_{\tau}^j|\delta^{-|k-j|}\frac{\delta^{-(j\wedge
k)}}{(\delta^{-(j\wedge
k)}+|y-y_{Q_{\tau}^j}|)^{n+1}}\big|q_{\delta^{-2j}}(L)(f)(y_{Q_{\tau}^j})\big|\Big|^2{dyds\over s^{n+1}}
\Big\}^{1\over 2}\Big\|_{L^1(\mathbb{R}^n)}\\
&\leq(\ln\delta)^{3\over 2}\Big\|\Big\{\sum_k  \Big| \sum\limits_j\sum_{\tau\in
I_j}|Q_{\tau}^j|\delta^{-|k-j|}\frac{\delta^{-(j\wedge
k)}}{(\delta^{-(j\wedge
k)}+|x-y_{Q_{\tau}^j}|)^{n+1}}\big|q_{\delta^{-2j}}(L)(f)(y_{Q_{\tau}^j})\big|\Big|^2
\Big\}^{1\over 2}\Big\|_{L^1(\mathbb{R}^n)}.
 \end{align*}

Next we claim that
\begin{eqnarray}\label{claim FJ}
\lefteqn{\sum_{\tau\in
I_j}|Q_{\tau}^j|\frac{\delta^{-(j\wedge
k)}}{(\delta^{-(j\wedge
k)}+|x-y_{Q_{\tau}^j}|)^{n+1}}\big|q_{\delta^{-2j}}(L)(f)(y_{Q_{\tau}^j})\big|}\\
 &\leq&
C  \delta^{ Mn({1\over r}-1) }\delta^{[j-(k\wedge j)]n({1\over r}-1)}\Big\{ \mathcal{M}\Big(
\sum_{\tau\in
I_j}\big|q_{\delta^{-2j}}(L)(f)(y_{Q_{\tau}^j})\big|^r\chi_{Q_{\tau}^j}(\cdot)
\Big) \Big\}^{1/r}(x),\nonumber
\end{eqnarray}
where $\mathcal{M}$ is the Hardy--Littlewood maximal function and
$ {n\over n+1}<r<1$.

We now prove this claim (\ref{claim FJ}).  We point out that this type of inequality is first proved by
Frazier and Jawerth in the Euclidean setting (See \cite{FJ}, pp.147--148).

To prove (\ref{claim FJ}), we first point out that for all $0<r<1$, $ \sum_j |a_j|\leq  \big( \sum_j |a_j|^r \big)^{1/r}$.
As a consequence, the left-hand side of the inequality (\ref{claim FJ}) is controlled by
\begin{eqnarray*}
&&\hskip-1cm\Big(\sum_{\tau\in
I_j}|Q_{\tau}^j|^{r}\frac{\delta^{-(j\wedge
k)r}}{(\delta^{-(j\wedge
k)}+|x-y_{Q_{\tau}^j}|)^{(n+1)r}}\big|q_{\delta^{-2j}}(L)(f)(y_{Q_{\tau}^j})\big|^r \Big)^{1\over r}\\
 &\leq& \Big(\sum_{\tau\in
A_0}|Q_{\tau}^j|^r\frac{\delta^{-(j\wedge
k)r}}{(\delta^{-(j\wedge
k)}+|x-y_{Q_{\tau}^j}|)^{(n+1)r}}\big|q_{\delta^{-2j}}(L)(f)(y_{Q_{\tau}^j})\big|^r \\
&&+ \sum_{\ell\geq1}\sum_{\tau\in
A_\ell}|Q_{\tau}^j|^r\frac{\delta^{-(j\wedge
k)r}}{(\delta^{-(j\wedge
k)}+|x-y_{Q_{\tau}^j}|)^{(n+1)r}}\big|q_{\delta^{-2j}}(L)(f)(y_{Q_{\tau}^j})\big|^r \Big)^{1\over r},
\end{eqnarray*}
where
\begin{eqnarray*}
 A_0&:=&\lbrace  \tau \in I_j:  |x-y_{Q_{\tau}^j}|  \leq \delta^{-(j\wedge k)}   \rbrace;\\
 A_\ell&:=&\lbrace  \tau \in I_j:     2^{\ell-1} \delta^{-(j\wedge k)}  < |x-y_{Q_{\tau}^j}| \leq  2^\ell\delta^{-(j\wedge k)}   \rbrace, \ \ \ell\geq 1,
 \end{eqnarray*}

Then the left-hand side of (\ref{claim FJ}) is bounded by
\begin{eqnarray*}
&&\hskip-.7cm\Big(\sum_{\tau\in
A_0}{|Q_{\tau}^j|^r\over \delta^{-(j\wedge
k)nr} }\big|q_{\delta^{-2j}}(L)(f)(y_{Q_{\tau}^j})\big|^r \\
&&+ \sum_{\ell\geq1}\sum_{\tau\in
A_\ell} { |Q_{\tau}^j|^r  \over  2^{(\ell-1) n r}\delta^{-(j\wedge
k)nr }  }  \frac{\delta^{-(j\wedge
k)r}}{(\delta^{-(j\wedge
k)}+ 2^{\ell-1} \delta^{-(j\wedge
k)}  )^{r}}\big|q_{\delta^{-2j}}(L)(f)(y_{Q_{\tau}^j})\big|^r \Big)^{1\over r}\\
&& \leq \Big(\sum_{\tau\in
A_0} {\delta^{-(j\wedge
k)n(1-r)}  \over |Q_{\tau}^j|^{1-r}  }  {|Q_{\tau}^j|\over \delta^{-(j\wedge
k)n} }  \big|q_{\delta^{-2j}}(L)(f)(y_{Q_{\tau}^j})\big|^r \\
&&\hskip.5cm+ \sum_{\ell\geq1} 2^{-\ell r}\sum_{\tau\in
A_\ell} {  2^{(\ell-1) n (1-r)}\delta^{-(j\wedge
k)n(1-r) }  \over |Q_{\tau}^j|^{1-r}   }   { |Q_{\tau}^j|  \over  2^{(\ell-1) n }\delta^{-(j\wedge
k)n }  } \big|q_{\delta^{-2j}}(L)(f)(y_{Q_{\tau}^j})\big|^r \Big)^{1\over r}\\
&& \leq C\delta^{ [j-(j\wedge k)]n({1\over r}-1) }\delta^{ Mn({1\over r}-1) } \Big(  {1\over |B(x,\delta^{-(j\wedge
k)})|} \int_{B(x,\delta^{-(j\wedge
k)})} \sum_{\tau\in
A_0}     \big|q_{\delta^{-2j}}(L)(f)(y_{Q_{\tau}^j})\big|^r  \chi_{Q_{\tau}^j}(y)dy \\
&&\hskip.5cm+ \sum_{\ell\geq1} 2^{-\ell (r-n (1-r))}  {1\over |B(x,2^\ell\delta^{-(j\wedge
k)})|} \int_{B(x,2^\ell\delta^{-(j\wedge
k)})}  \sum_{\tau\in
A_\ell}  \big|q_{\delta^{-2j}}(L)(f)(y_{Q_{\tau}^j})\big|^r  \chi_{Q_{\tau}^j}(y)dy \Big)^{1\over r}\\
&&\leq C \delta^{ Mn({1\over r}-1) } \delta^{[j-(k\wedge j)]n({1\over r}-1)}\Big\{ \mathcal{M}\Big(
\sum_{\tau\in
I_j}\big|q_{\delta^{-2j}}(L)(f)(y_{Q_{\tau}^j})\big|^r\chi_{Q_{\tau}^j}(\cdot)
\Big) \Big\}^{1/r}(x),
\end{eqnarray*}
where in the last inequality we use the fact that ${n\over n+1}<r$. This implies that the claim (\ref{claim FJ}) holds.

As a consequence, we can obtain that
\begin{align*}
\|T_\delta(f)\|_{H_L^1(\mathbb{R}^n)}
&\leq C\delta^{ Mn({1\over r}-1) }(\ln\delta)^{3\over 2}\Big\|\Big\{\sum_k \Big|
\sum\limits_j\delta^{-|k-j|}\delta^{[j-(k\wedge j)]n({1\over r}-1)}\\
&\quad\times\Big\{ \mathcal{M}\Big( \sum_{\tau\in
I_j}\big|q_{\delta^{-2j}}(L)(f)(y_{Q_{\tau}^j})\big|^r\chi_{Q_{\tau}^j}(\cdot) \Big)
\Big\}^{1/r}(x)\Big|^2 \Big\}^{1\over
2}\Big\|_{L^1(\mathbb{R}^n)}\\
&\leq C\delta^{Mn({1\over r}-1) }(\ln\delta)^{3\over 2}\Big\|\Big\{\sum_k
\Big(\sum\limits_j\delta^{-|k-j|}\delta^{[j-(k\wedge j)]n({1\over r}-1)}\Big)\\
&\quad\times\Big(\sum\limits_j\delta^{-|k-j|}\delta^{[j-(k\wedge j)]n({1\over r}-1)}\Big\{ \mathcal{M}\Big( \sum_{\tau\in
I_j}\big|q_{\delta^{-2j}}(L)(f)(y_{Q_{\tau}^j})\big|^r\chi_{Q_{\tau}^j}(\cdot) \Big)
\Big\}^{2/r}(x)\Big) \Big\}^{1\over
2}\Big\|_{L^1(\mathbb{R}^n)}
\end{align*}

Observe that
\begin{align}\label{e4.10-s}
\sum\limits_j\delta^{-|k-j|}\delta^{[j-(k\wedge j)]n({1\over r}-1)}+\sum\limits_k\delta^{-|k-j|}\delta^{[j-(k\wedge j)]n({1\over r}-1)}\leq \frac{C(n,r)}{\delta-1}.
\end{align}
It follows that
\begin{align*}
\|T_\delta(f)\|_{H_L^1(\mathbb{R}^n)}&\leq C\delta^{ Mn({1\over r}-1) }(\ln\delta)^{3\over 2}(\delta-1)^{-{1\over 2}}\Big\|\Big\{\sum_k
\sum\limits_j\delta^{-|k-j|}\delta^{[j-(k\wedge j)]n({1\over r}-1)}\\
&\quad\times\Big\{ \mathcal{M}\Big( \sum_{\tau\in
I_j}\big|q_{\delta^{-2j}}(L)(f)(y_{Q_{\tau}^j})\big|^r\chi_{Q_{\tau}^j}(\cdot) \Big)
\Big\}^{2/r}(x) \Big\}^{1\over 2}\Big\|_{L^1(\mathbb{R}^n)}\\
&\leq C\delta^{ Mn({1\over r}-1) }\frac{(\ln\delta)^{3\over 2}}{(\delta-1)}\Big\|\Big\{ \sum_j \Big\{
\mathcal{M}\Big( \sum_{\tau\in
I_j}\big|q_{\delta^{-2j}}(L)(f)(y_{Q_{\tau}^j})\big|^r\chi_{Q_{\tau}^j}(\cdot) \Big)
\Big\}^{2/r}(x) \Big\}^{1\over 2}\Big\|_{L^1(\mathbb{R}^n)}\\
&\leq C\delta^{ Mn({1\over r}-1) }(\ln\delta)^{1\over2}\Big\|\Big\{ \sum_j \sum_{\tau\in
I_j}\big|q_{\delta^{-2j}}(L)(f)(y_{Q_{\tau}^j})\big|^2\chi_{Q_{\tau}^j}(x)
\Big\}^{1\over 2}\Big\|_{L^1(\mathbb{R}^n)}\\
&=C\delta^{ Mn({1\over r}-1) }\Big\|g_{1,\delta}(f)\Big\|_{L^1(\mathbb{R}^n)},
\end{align*}
where the last inequality follows from boundedness of the
vector-valued maximal function.

Now applying (\ref{e1 lemma H1 to L1 molecule}) in Lemma \ref{lemma H1 to L1 molecule}, we  obtain that
\begin{eqnarray*}
\|T_\delta(f)\|_{H_L^1(\mathbb{R}^n)}\leq
C\delta^{ Mn({1\over r}-1) }\|f\|_{H_L^1(\mathbb{R}^n)}.
\end{eqnarray*}
The proof of this theorem is complete.
\end{proof}

\smallskip

\begin{theorem}\label{theorem error term of H1}
Define $R_\delta(f)(x)=f(x)-T_\delta(f)(x)$. Then there exists positive
constants $1<\delta<2$ and $M\geq1$ such that $\|R_\delta(f)\|_{H_L^1(\mathbb{R}^n)}\leq
\frac{1}{2}\|f\|_{H_L^1(\mathbb{R}^n)}$.
\end{theorem}
\begin{proof}
We decompose $R_\delta(f)$ into four terms:
$R(f)=R_{\delta,1}(f)+R_{\delta,2}(f)+R_{\delta,3}(f)+R_{\delta,4}(f)$, where
\begin{eqnarray*}
R_{\delta,1}(f)&:=&\sum_j\sum_{\tau\in I_j}\int_{\delta^{-j}}^{\delta^{-j+1}}
\int_{Q_{\tau}^j}[q_{t^2}(x,y)-q_{t^2}(x,y_{Q_{\tau}^j})]q_{t^2}(L)f(y)dy\frac{dt}{t};\\
R_{\delta,2}(f)&:=&\sum\limits_j\sum_{\tau\in
I_j}\int_{\delta^{-j}}^{\delta^{-j+1}}
\int_{Q_{\tau}^j}[q_{t^2}(x,y_{Q_{\tau}^j})-q_{j}(x,y_{Q_{\tau}^j})]q_{t^2}(L)f(y)dy\frac{dt}{t};\\
R_{\delta,3}(f)&:=&\sum\limits_j\sum_{\tau\in
I_j}\int_{\delta^{-j}}^{\delta^{-j+1}}
\int_{Q_{\tau}^j}q_{j}(x,y_{Q_{\tau}^j})[q_{t^2}(L)f(y)-q_{t^2}(L)f(y_{Q_{\tau}^j})]dy\frac{dt}{t};\\
R_{\delta,4}(f)&:=&\sum\limits_j\sum_{\tau\in
I_j}\int_{\delta^{-j}}^{\delta^{-j+1}}
\int_{Q_{\tau}^j}q_{j}(x,y_{Q_{\tau}^j})[q_{t^2}(L)f(y_{Q_{\tau}^j})-q_{\delta^{-2j}}(L)f(y_{Q_{\tau}^j})]dy\frac{dt}{t}.
\end{eqnarray*}
We first estimate $R_{\delta,1}(f)$. For every $f\in H_L^1(\mathbb{R}^n)\cap L^2(\RN)$, from the definition of $R_{\delta,1}(f)$, we
have
\begin{align}\label{R1 f}
&\|R_{\delta,1}(f)\|_{H_L^1(\mathbb{R}^n)}\\
&=\Big\|\Big\{ \int_{\Gamma(x)} \big| q_{s^2}(L)\big(R_{\delta,1}(f)\big)(y)\big|^2 {dyds\over s^{n+1}}
\Big\}^{1\over
2}\Big\|_{L^1(\mathbb{R}^n)} \nonumber\\
&= \Big\|\Big\{\sum_k \int_{\delta^{-k}}^{\delta^{-k+1}} \int_{|x-y|<s} \Big| \sum_j\sum_{\tau\in
I_j}\nonumber\\
& \hskip1cm\int_{\delta^{-j}}^{\delta^{-j+1}}
\int_{Q_{\tau}^j}[q_{s^2}(L)q_{t^2}(L)(y,z)-q_{s^2}(L)q_{t^2}(L)(y,y_{Q_{\tau}^j})]q_{t^2}(L)f(z)dz\frac{dt}{t}\Big|^2{dyds\over s^{n+1}}
\Big\}^{1\over
2}\Big\|_{L^1(\mathbb{R}^n)}\nonumber
\end{align}
Note that for $s\in (\delta^{-k} ,  \delta^{-k+1})$ and $t\in (\delta^{-j} ,  \delta^{-j+1})$, we have
\begin{eqnarray*}
&&|q_{s^2}(L)q_{t^2}(L)(y,z)-q_{s^2}(L)q_{t^2}(L)(y,y_{Q_{\tau}^j})|\\
&&\hskip.5cm=\Big|\int_{\mathbb{R}^n}\int_{\mathbb{R}^n}
q_{s^2}(y,w)\varphi_{t^2}(w,v)\big[\zeta_{t^2}(v,z)-\zeta_{t^2}(v,y_{Q_{\tau}^j})\big] dwdv
\Big|\\
&&\hskip.5cm\leq \int_{\mathbb{R}^n}
|q_{s^2}(L)\varphi_{t^2}(L)(y,v)|\big|\zeta_{t^2}(v,z)-\zeta_{t^2}(v,y_{Q_{\tau}^j})\big| dv\\
&&\hskip.5cm\leq C \delta^{-|j-k|}\frac{|z-y_{Q_{\tau}^j}|}{\delta^{-j}}\int_{\RN}\frac{\delta^{-(j\wedge k)}}{(\delta^{-(j\wedge k)}+|y-v|)^{n+1}}\frac{\delta^{-j}}{(\delta^{-j}+|v-y_{Q_{\tau}^j}|)^{n+1}}\,dv\\
&&\hskip.5cm\leq C \delta^{-|j-k|}\frac{|z-y_{Q_{\tau}^j}|}{\delta^{-j}}\frac{\delta^{-(j\wedge k)}}{(\delta^{-(j\wedge k)}+|y-y_{Q_{\tau}^j}|)^{n+1}},
\end{eqnarray*}
where in the second inequality above we have used the similar argument of (\ref{s-e2.1}) and (\ref{s-e2.2}).
Thus, by substituting the above estimate into the right-hand side of \eqref{R1 f}, we get
\begin{align*}
\|R_{\delta,1}(f)\|_{H_L^1(\mathbb{R}^n)}
 &\leq C\Big\|\Big\{\sum_k \int_{\delta^{-k}}^{\delta^{-k+1}} \Big|
\sum_j\sum_{\tau\in I_j}\int_{\delta^{-j}}^{\delta^{-j+1}}
\int_{Q_{\tau}^j} \delta^{-|j-k|}\frac{|z-y_{Q_{\tau}^j}|}{\delta^{-j}}\\
&\quad\times\frac{\delta^{-(j\wedge k)}}{(\delta^{-(j\wedge k)}+|x-y_{Q_{\tau}^j}|)^{n+1}}
|q_{t^2}(L)f(z)|dz\frac{dt}{t}\Big|^2 {ds\over s}\Big\}^{1\over
2}\Big\|_{L^1(\mathbb{R}^n)}\\
&\leq C\delta^{-M}\sqrt{\ln\delta}\Big\|\Big\{\sum_k \Big|
\sum_j\delta^{-|k-j|}\sum_{\tau\in I_j} |Q_{\tau}^j|
\frac{\delta^{-(j\wedge k)}}{(\delta^{-(j\wedge k)}+|x-y_{Q_{\tau}^j}|)^{n+1}}
\\
&\quad \times \int_{\delta^{-j}}^{\delta^{-j+1}} {1\over |Q_{\tau}^j|} \int_{Q_{\tau}^j}|q_{t^2}(L)f(z)|dz\frac{dt}{t}
\Big|^2 \Big\}^{1\over
2}\Big\|_{L^1(\mathbb{R}^n)}\\
&\leq C \delta^{-M}\sqrt{\ln\delta}\Big\|\Big\{\sum_k \Big|
\sum_j\delta^{-|k-j|}\delta^{ Mn({1\over r}-1) } \delta^{[j-(k\wedge j)]n({1\over r}-1)}\\
&\quad\times \Big [\mathcal{M}\Big(\sum_{\tau\in I_j}
\Big|\int_{\delta^{-j}}^{\delta^{-j+1}}{1\over |Q_{\tau}^j|}\int_{Q_{\tau}^j}|q_{t^2}(L)f(z)|dz\frac{dt}{t}\Big|^r
\ \chi_{Q_{\tau}^j}(\cdot)  \Big)(x)\Big ]^{1/r} \Big|^2
\Big\}^{1\over
2}\Big\|_{L^1(\mathbb{R}^n)},
\end{align*}
where the last inequality follows from the claim (\ref{claim FJ}) and
$r$ can be any number in $ ({n\over n+1},1)$. Then using H\"older's inequality, we have
$$\Big| \int_{\delta^{-j}}^{\delta^{-j+1}}{1\over |Q_{\tau}^j|}\int_{Q_{\tau}^j}|q_{t^2}(L)f(z)|dz\frac{dt}{t}\Big|\leq
(\ln\delta)^{1\over2}\Big(\int_{\delta^{-j}}^{\delta^{-j+1}}{1\over |Q_{\tau}^j|}\int_{Q_{\tau}^j}|q_{t^2}(L)f(z)|^2dz\frac{dt}{t}\Big)^{1\over2}.
$$
Thus,
\begin{align*}
\|R_{\delta,1}(f)\|_{H_L^1(\mathbb{R}^n)}
&\leq C \delta^{-M}\delta^{ Mn({1\over r}-1) }(\ln\delta)\Big\|\Big\{\sum_k \Big|
\sum_j\delta^{-|k-j|}\delta^{[j-(k\wedge j)]n({1\over r}-1)}\\
&\quad\times \Big [\mathcal{M}\Big(\sum_{\tau\in I_j}
\Big|\int_{\delta^{-j}}^{\delta^{-j+1}}{1\over
|Q_{\tau}^j|}\int_{Q_{\tau}^j}|q_{t^2}(L)f(z)|^2dz\frac{dt}{t}\Big|^{r\over 2} \
\chi_{Q_{\tau}^j}(\cdot) \Big)(x)\Big ]^{1/r} \Big|^2
\Big\}^{1\over
2}\Big\|_{L^1(\mathbb{R}^n)}\\
&\leq C \delta^{-M}\delta^{ Mn({1\over r}-1) }(\ln\delta) \frac{1}{\delta-1}\\
&\quad\times\Big\|\Big \{\sum_j
\Big [\mathcal{M}\Big(\sum_{\tau\in I_j}
\Big|\int_{\delta^{-j}}^{\delta^{-j+1}}{1\over
|Q_{\tau}^j|}\int_{Q_{\tau}^j}|q_{t^2}(L)f(z)|^2dz\frac{dt}{t}\Big|^{r\over 2} \
\chi_{Q_{\tau}^j}(\cdot) \Big)(x)\Big ]^{2/r}  \Big\}^{1\over
2}\Big\|_{L^1(\mathbb{R}^n)}\\
&\leq C \delta^{-M}\delta^{ Mn({1\over r}-1) }\Big\|\Big\{ \sum_j \sum_{\tau\in I_j}
\int_{\delta^{-j}}^{\delta^{-j+1}}{1\over
|Q_{\tau}^j|}\int_{Q_{\tau}^j}|q_{t^2}(L)f(z)|^2dz\frac{dt}{t} \
\chi_{Q_{\tau}^j}(x)
  \Big\}^{1\over
2}\Big\|_{L^1(\mathbb{R}^n)},
\end{align*}
where in the second inequality above we have used (\ref{e4.10-s}) and in the last inequality we have used the boundedness of the vector-valued maximal function.

Now applying  Lemma \ref{lemma H1 to L1 molecule}, we can obtain that
\begin{eqnarray*}
\|R_{\delta,1}(f)\|_{H_L^1(\mathbb{R}^n)}
\leq C\delta^{-M}\delta^{ Mn({1\over r}-1) } \|f\|_{H_L^1(\mathbb{R}^n)}.
\end{eqnarray*}

Next we only need to estimate $R_{\delta,4}(f)$ since the terms $R_{\delta,2}(f)$ and
$R_{\delta,3}(f)$ can be obtained by following similar steps as in $R_{\delta,1}(f)$
and $R_{\delta,4}(f)$, respectively.

For any $f\in H_L^1(\mathbb{R}^n)$, we have
\begin{align*}
\|R_{\delta,4}(f)\|_{H_L^1(\mathbb{R}^n)}
&=\Big\|\Big\{\int_{\Gamma(x)} \big| q_{s^2}\big(R_{\delta,4}(f)\big)(y)\big|^2 {dyds\over s^{n+1}}
\Big\}^{1\over
2}\Big\|_{L^1(\mathbb{R}^n)}\\
&=\Big\|\Big\{\sum_k \int_{\delta^{-k}}^{\delta^{-k+1}} \int_{|x-y|<s} \\
&\quad\quad\Big| \sum\limits_j\sum_{\tau\in
I_j}\int_{\delta^{-j}}^{\delta^{-j+1}}
\int_{Q_{\tau}^j}q_{s^2}(L)q_{\delta^{-2j}}(L)(y,y_{Q_{\tau}^j})[q_{t^2}f(y_{Q_{\tau}^j})-q_{j}f(y_{Q_{\tau}^j})]dz\frac{dt}{t}
\Big|^2 {dyds\over s^{n+1}}\Big\}^{1\over 2}\Big\|_{L^1(\mathbb{R}^n)}.
\end{align*}
 Applying \eqref{e2.5} we obtain that
\begin{align*}
\|R_{\delta,4}(f)\|_{H_L^1(\mathbb{R}^n)}
&\leq  C\ln \delta\Big\|\Big\{\sum_k \Big|
\sum\limits_j\sum_{\tau\in I_j}\int_{\delta^{-j}}^{\delta^{-j+1}}
\int_{Q_{\tau}^j}
\delta^{-|j-k|}\\
&\hskip.7cm\times\frac{\delta^{-(j\wedge k)}}{(\delta^{-(j\wedge k)}+|x-y_{Q_{\tau}^j}|)^{n+1}}
\Big(\int_{\delta^{-j}}^{ \delta^{-j+1} } |  r^2Lq^{(1)}(r^2L)f(y_{Q_{\tau}^j})|^2\frac{dr}{r}\Big)^{1/2}
dy\frac{dt}{t} \Big|^2 \Big\}^{1\over
2}\Big\|_{L^1(\mathbb{R}^n)}\\
 &\leq  C(\ln\delta)^{2}\Big\|\Big\{\sum_k \Big|
\sum\limits_j\delta^{-|j-k|} \sum_{\tau\in I_j} |Q_{\tau}^j|
\\
&\hskip.7cm\times\frac{\delta^{-(j\wedge k)}}{(\delta^{-(j\wedge k)}+|x-y_{Q_{\tau}^j}|)^{n+1}}
\Big(\int_{\delta^{-j}}^{\delta^{-j+1}}|r^2Lq^{(1)}(r^2L)f(y_{Q_{\tau}^j})|^2\frac{dr}{r}\Big)^{1/2}\Big|^2 \Big\}^{1\over 2}\Big\|_{L^1(\mathbb{R}^n)}\\
&\leq  C(\ln\delta)^{2}\Big\|\Big\{\sum_k \Big|
\sum\limits_j
\delta^{-|j-k|} \delta^{ Mn({1\over r}-1) } \delta^{[j-(k\wedge j)]n({1\over r}-1)} \\
&\hskip.7cm\times \Big\{ \mathcal{M} \Big( \sum_{\tau\in I_j}
\Big(\int_{\delta^{-j}}^{\delta^{-j+1}}|r^2Lq^{(1)}(r^2L)f(y_{Q_{\tau}^j})|^2\frac{dr}{r}\Big)^{r/2}
\chi_{Q_{\tau}^j}(\cdot)\Big)(x)\Big\}^{1/r}
\Big|^2 \Big\}^{1\over 2}\Big\|_{L^1(\mathbb{R}^n)}\\
&\leq  C (\ln\delta)^{2}\delta^{ Mn({1\over r}-1) } \Big\|\Big\{\sum_k
\sum\limits_j
\delta^{-|j-k|} \delta^{[j-(k\wedge j)]n({1\over r}-1)}\\
&\hskip.7cm\times \Big\{ \mathcal{M} \Big( \sum_{\tau\in I_j}
\Big(\int_{\delta^{-j}}^{\delta^{-j+1}}|r^2Lq^{(1)}(r^2L)f(y_{Q_{\tau}^j})|^2\frac{dr}{r}\Big)^{r/2}
\chi_{Q_{\tau}^j}(\cdot)\Big)(x)\Big\}^{2/r}
 \Big\}^{1\over 2}\Big\|_{L^1(\mathbb{R}^n)}\\
&\leq  C (\ln\delta) \delta^{ Mn({1\over r}-1) }\Big\|\Big\{ \sum_j \Big\{ \mathcal{M} \Big(
\sum_{\tau\in I_j}
\Big(\int_{\delta^{-j}}^{\delta^{-j+1}}|r^2Lq^{(1)}(r^2L)f(y_{Q_{\tau}^j})|^2\frac{dr}{r}\Big)^{r/2}
\chi_{Q_{\tau}^j}(\cdot)\Big)\Big\}^{2/r}(x)
 \Big\}^{1\over 2}\Big\|_{L^1(\mathbb{R}^n)}\\
&\leq  C (\ln\delta)\delta^{ Mn({1\over r}-1) } \Big\|\Big\{ \sum_j
\sum_{\tau\in I_j}
\int_{\delta^{-j}}^{\delta^{-j+1}}| r^2Lq^{(1)}(r^2L)f(y_{Q_{\tau}^j})|^2\frac{dr}{r}
\chi_{Q_{\tau}^j}(x)
 \Big\}^{1\over 2}\Big\|_{L^1(\mathbb{R}^n)}\\
&=  C (\ln\delta)\delta^{ Mn({1\over r}-1) } \Big\|g_{4,\delta}(f)\Big\|_{L^1(\mathbb{R}^n)}.
\end{align*}

By applying Lemma
\ref{lemma H1 to L1 molecule}, we can obtain that
\begin{eqnarray*}
\|R_{\delta,4}(f)\|_{H_L^1(\mathbb{R}^n)} \leq
C(\delta-1) \delta^{ Mn({1\over r}-1)}\|f\|_{H_L^1(\mathbb{R}^n)}.
\end{eqnarray*}

Similarly we can verify the $H_L^1(\mathbb{R}^n)$-norm of the terms $R_{\delta,2}(f)$ and $R_{\delta,3}(f)$ with constants $ C\delta^{ Mn({1\over r}-1) }(\delta-1) $ and $C \delta^{-M}\delta^{ Mn({1\over r}-1) }$, respectively.

Thus, combining the estimates of $R_{\delta,1}(f)$, $R_{\delta,2}(f)$, $R_{\delta,3}(f)$ and
$R_{\delta,4}(f)$, we have, there exists a constant $C_2>1$ such that
\begin{align*}
\|R_\delta(f)\|_{H_L^1(\mathbb{R}^n)}&\leq
C_2\Big(\delta^{ Mn({1\over r}-1) }(\delta-1) +\delta^{-M(1-n({1\over r}-1))}\Big)\|f\|_{H_L^1(\mathbb{R}^n)}\\
&\leq C_2\Big(\delta^{M}(\delta-1) +\delta^{-M(1-n({1\over r}-1))}\Big)\|f\|_{H_L^1(\mathbb{R}^n)},
\end{align*}
since $n/(n+1)<r<1$.  Let us  choose  $\delta$ close to 1, such that
$$ 2C_2\Big(4C_2\Big)^{\frac{1}{1-n({1\over r}-1)}} (\delta-1)< \frac{1}{4}
$$
and choose $\tilde{M}\in {\mathbb R}^+$ such that $C_2 \delta^{-\tilde{M}(1-n({1\over r}-1))}= \frac{1}{4}$. Let $M=[\tilde{M}]+1$, where we use $[x]$ to denote the maximal integer which is not greater than $x$. It follows that
$2^{-(3-n(\frac{1}{r}-1))}\leq C_2 \delta^{-M(1-n({1\over r}-1))}\leq \frac{1}{4}$. It implies that $\delta^{ M} \leq 2(4C_2)^{\frac{1}{1-n({1\over r}-1)}}$, thus
$C_2\delta^{M}(\delta-1)\leq \frac{1}{4}$. It follows that $\|R_\delta(f)\|_{H_L^1(\mathbb{R}^n)}\leq  \frac{1}{2}\|f\|_{H_L^1(\mathbb{R}^n)}$.
\end{proof}

\smallskip

We now apply the results in Theorems \ref{prop of S bd on H1} and \ref{theorem error term of H1} to
prove our main result, Theorem \ref{theorem  wavelet expansion of H1}.

\smallskip

\begin{proof}[Proof of Theorem \ref{theorem  wavelet expansion of H1}]
By Theorem \ref{theorem error term of H1}, we have that for every
$k=0,1,\cdots$, \ \ $\|R_\delta^k(f)\|_{H_L^1(\mathbb{R}^n)}\leq
\frac{1}{2^k}\|f\|_{H_L^1(\mathbb{R}^n)}$. Therefore,  the operator
$T_\delta$ as defined in \eqref{Tdelta} is
invertible, and
\begin{eqnarray}
\big\|T_\delta^{-1}f\big\|_{H_L^1(\mathbb{R}^n)}\leq
\sum\limits_{k=0}^{+\infty}\big\|R_\delta^k(f)\big\|_{H_L^1(\mathbb{R}^n)}\leq
2\big\|f\big\|_{H_L^1(\mathbb{R}^n)}.
\end{eqnarray}
Thus, for every $f\in H_L^1(\mathbb{R}^n)$, by the definition of
operator $T_\delta$, one can write
\begin{eqnarray}\label{H1 expansion}
f=T_\delta\circ T_\delta^{-1}f=\sum_j\sum_{\tau\in I_j}\langle T_\delta^{-1}f,
 {\psi}^*_{j,\tau}\rangle  \psi_{j,\tau}  \hskip 2cm  in \ \ H_L^1(\mathbb{R}^n).
\end{eqnarray}

Applying Lemma \ref{lemma H1 to L1 molecule}, we have
\begin{align*}
&\Big\|\Big( \sum\limits_j \sum_{\tau\in I_j}\big|\langle T_\delta^{-1}f,
{\psi}^*_{j,\tau}\rangle\big|^2
\frac{1}{|Q_{\tau}^j|}\chi_{Q_{\tau}^j}\Big)^{1\over
2}\Big\|_{L^1(\mathbb{R}^n)}\\
 &\quad\leq
\sqrt{\ln\delta}\Big\|\Big( \sum\limits_j
\sum_{\tau\in I_j}
|q_{\delta^{-2j}}(L)(T_\delta^{-1}f)(y_{Q_{\tau}^j})|^2\chi_{Q_{\tau}^j}\Big)^{1\over
2}\Big\|_{L^1(\mathbb{R}^n)}\\
&\quad\leq C\big\|T_\delta^{-1}f\big\|_{H_L^1(\mathbb{R}^n)}\leq
C\big\|f\big\|_{H_L^1(\mathbb{R}^n)}.
\end{align*}

For the left inequality in (\ref{wavelet norm of H1}), using
(\ref{H1 expansion}) we have
\begin{align*}
&\|f\|_{H_L^1(\mathbb{R}^n)}=
\Big\|\Big\{\int_0^\infty \!\!\int_{|x-y|<s} |q_{s^2}(L)(f)(y)|^2 dy{ds\over {s^{n+1}}}\Big\}^{1/2}\Big\|_{L^1(\mathbb{R}^n)}\\
&=\sqrt{\ln\delta}\Big\|\Big\{ \sum_k\int_{\delta^{-k}}^{\delta^{-k+1}}\!\!\int_{|x-y|<s} \big|    \sum_j\sum_{\tau\in I_j}\langle
T_\delta^{-1}f,
 {\psi}^*_{j,\tau}\rangle |Q_{\tau}^j|^{1/2} q_{s^2}q_{j}(y,y_{Q_{\tau}^j})  \Big|^2 dy {ds\over {s^{n+1}}} \Big\}^{1/2}\Big\|_{L^1(\mathbb{R}^n)}\\
&\leq C\ln\delta\Big\|\Big\{\sum_k\Big|    \sum_j
\delta^{-|j-k|}\\
&\quad\times\sum_{\tau\in I_j} |Q_{\tau}^j|
\frac{\delta^{-(j\wedge k)} } {(\delta^{-(j\wedge k)}+|x-y_{Q_{\tau}^j}|)^{n+1}}
|\langle T_\delta^{-1}f,
 |Q_{\tau}^j|^{-1/2}\psi^*_{j,\tau}\rangle| \Big|^2\Big\}^{1/2}\Big\|_{L^1(\mathbb{R}^n)}\\
&\leq C\ln\delta\Big\|\Big\{\sum_k\Big|    \sum_j
  \delta^{-|j-k|}  \delta^{Mn({1\over r}-1)} \delta^{[j-(j\wedge k)]n({1\over r}-1)}\\
&\quad\times \Big\{\mathcal{M}\Big(\sum_{\tau\in I_j} \big|\langle
T_\delta^{-1}f,
 |Q_{\tau}^j|^{-1/2}\psi^*_{j,\tau}\rangle\big|^r\chi_{Q_{\tau}^j}(\cdot)\Big)(x)\Big\}^{1/r}       \Big|^2\Big\}^{1/2}\Big\|_{L^1(\mathbb{R}^n)}\\
&\leq C\frac{\ln\delta}{\delta-1}\ \delta^{Mn({1\over r}-1)}\Big\|\Big\{\sum_j
 \Big\{\mathcal{M}\Big(\sum_{\tau\in I_j} |\langle T_\delta^{-1}f,
 |Q_{\tau}^j|^{-1/2}\psi^*_{j,\tau}\rangle|^r\chi_{Q_{\tau}^j}(\cdot)\Big)\Big\}^{2/r}(x)       \Big\}^{1/2}\Big\|_{L^1(\mathbb{R}^n)}\\
&\leq C\delta^{Mn({1\over r}-1)}\Big\|\Big\{\sum_j
 \sum_{\tau\in I_j} |\langle T_\delta^{-1}f,
 |Q_{\tau}^j|^{-1/2}\psi^*_{j,\tau}\rangle|^2\chi_{Q_{\tau}^j}(x)
 \Big\}^{1/2}\Big\|_{L^1(\mathbb{R}^n)}.
\end{align*}
This completes the proof of Theorem \ref{theorem  wavelet expansion
of H1}.
\end{proof}

\vspace{0.2cm}
\section{Application: A maximal function characterization of $H_L^1(\mathbb{R}^n)$  }
 \setcounter{equation}{0}

In this section, we continue the discussion from Section 4 regarding a characterization of   the Hardy
space   $H_L^1(\mathbb{R}^n)$ in terms of  radial maximal function  under  the following conditions:

\smallskip

  \noindent
 ${\bf   (H1)'}$ \
  $L$ is a second order non-negative self-adjoint operator on $L^2(\mathbb{R}^n)$;

\smallskip

 \noindent
 ${\bf   (H2)'}$ \ The kernel of $e^{-tL}$, denoted by $p_t(x,y)$,
 is a measurable function on $\mathbb{R}^n\times \mathbb{R}^n$ and  satisfies
a Gaussian upper bound, that is
$$
\big|p_t(x,y)\big|\leq C t^{-n/2} \exp\left(-{  {|x-y|^2}\over  ct}\right)
$$
for all $t>0$,  and $x,y\in \mathbb{R}^n,$   where $C$ and $c$   are positive
constants.

The space $H_L^1(\mathbb{R}^n)$  involves some different
characterizations,  see for examples, \cite{AuDM, DL, DLY, DY2, HLMMY, HM, JY, SY1,SY2,YY}.
 If an operator $L$ satisfies conditions
${\bf (H1)'}$ and ${\bf (H2)'}$, then  for any $M\geq 1, 1<q\leq \infty,$
\begin{eqnarray}\label{e5.1}
f\in H_L^1(\mathbb{R}^n)
&\substack{{\rm (i)}\\ \Longleftrightarrow}& N_hf(x):=
\sup_{|y-x|<t}\big|e^{-t^2L}f(y)\big|\in L^1({\mathbb{R}}^n)\nonumber\\
&\substack{{\rm (ii)}\\  \Longleftrightarrow}&
 f {\rm \ has\ an }\  { (1, q, M)}  \ {\rm atomic \ decomposition}\
 f=\sum_{j=0}^{\infty}\lambda_j a_j  \ {\rm with\ }   \sum_{j=0}^{\infty}|\lambda_j|<\infty.
\end{eqnarray}
  Recall that
a function $a\in L^2(\R^n)$ is called an $(1,q,M)$-atom associated to
the operator $L$ if there exist a function $b\in {\mathcal D}(L^M)$
and a ball $B\subset \R^n$ such that
  (i)\ $a=L^M b$;
  (ii) \ {\rm supp}\  $L^{k}b\subset B, \ k=0, 1, \dots, M$;
(iii) \ $\|(r_B^2L)^{k}b\|_{L^q (\R^n)}\leq r_B^{2M}
|B|^{1/q-1/p},\ k=0,1,\dots,M$.

Next, consider the following radial maximal
operator  associated to the heat semigroup generated by the operator $L$,
\begin{eqnarray}\label{e5.2}
f^{+}_h(x)=\sup_{t>0}|e^{-t^2L}f(x)|.
\end{eqnarray}
Define the spaces $H^1_{L, {\rm max}}(\R^n)$
   as the completion of $L^2(\R^n)$ in the
norms given by the $L^1(\R^n)$ norm of the  maximal
function, i.e.,
$
\|f\|_{H^1_{L, {\rm max}}(\R^n)}=\|f^{+}_h\|_{L^1(\R^n)}.
$
From (i) of \eqref{e5.1},  the following
continuous inclusion holds:
\begin{eqnarray}\label{e5.3}
H_{L}^1(\mathbb{R}^n)   \subseteq H^1_{L, {\rm max}}(\R^n).
\end{eqnarray}

The aim of this section is to prove the following result.

\begin{theorem}\label{th5.1}
Suppose  that an operator $L$ satisfies ${\bf   (H1)'}$  and ${\bf   (H2)'}$. In addition, we assume that
the gradient estimate  of the heat kernel of the semigroup $\{e^{-tL}\}_{t>0}$
  satisfies the pointwise bound
$$\leqno{\bf (H3)'} \hspace{5cm}
 |\nabla p_t(x,y)|\leq C t^{-(n+1)/2} \exp\left(-{  {|x-y|^2}\over  ct}\right).
$$
Then we have that  $ H^1_{L, {\rm max}}(\R^n) \subseteq H_{L}^1(\mathbb{R}^n)$
and hence by \eqref{e5.3},
$$
 H^1_{L, {\rm max}}(\R^n) \simeq H_{L}^1(\mathbb{R}^n).
$$
\end{theorem}

\begin{remark} We should note that the equivalency of the radial maximal function characterization
 and the nontangential maximal function characterizations of $H_L^1(\mathbb{R}^n)$ have been obtained in \cite{YY,SY2}.
 Our Theorem~\ref{th5.1} provides  a different proof  by using  the frame decomposition.
\end{remark}

The proof of Theorem~\ref{th5.1} is based on the following lemma.

\begin{lemma}\label{le6.3} Suppose  that an operator $L$ satisfies ${\bf   (H1)'}$, ${\bf   (H2)'}$ and ${\bf   (H3)'}$.
For any $f\in H_{L,max}^1(\mathbb{R}^n)\cap L^2(\mathbb{R}^n)$, there exists a constant $C>0$ independent of $f$
such that
\begin{eqnarray}\label{e lemma maximal}
\left\| \sup_{|x-y|< 2t} |t\nabla e^{-t^2L} f(y)|  \right\|_{L^1(\mathbb{R}^n)}
\leq C \big\| f  \big\|_{H_{L}^1(\mathbb{R}^n)}.
\end{eqnarray}
\end{lemma}
\begin{proof}
First, we assume that   $f\in H_{L,max}^1(\mathbb{R}^n)\cap L^2(\mathbb{R}^n)$.
For any $x,y\in \mathbb{R}^n$ and $t>0$ with $|x-y|<t$, we apply Theorem \ref{main1} and
 Remark  1.5
with $k>n$ to obtain
\begin{eqnarray*}
&&\hspace{-0.6cm}|t\nabla e^{-t^2L} f(y)| =\left|   t\nabla e^{-t^2L}\Big(\sum\limits_j\sum\limits_{\tau\in I_j}\langle T_\delta^{-1}f,
 \psi_{j,\tau}\rangle  \psi_{j,\tau} \Big)(y) \right|\\
 &&=\ln\delta \left|   \sum\limits_j\sum\limits_{\tau\in I_j} |Q_{\tau}^j| \big\langle  T_\delta^{-1}f(\cdot),
 (\delta^{-2j}L)^k e^{-\delta^{-2j}L}(\cdot,y_{Q_{\tau}^j})\big\rangle
   t\nabla e^{-t^2L} (\delta^{-2j}L)^k e^{-\delta^{-2j}L}(y,y_{Q_{\tau}^j})  \right|.
\end{eqnarray*}
From the property of the semigroup $\{e^{-tL}\}_{t>0}$, we  can show that  for every $\eta>n$,
\begin{eqnarray*}
\left|t\nabla e^{-t^2L} (\delta^{-2j}L)^k e^{-\delta^{-2j}L}(y,y_{Q_{\tau}^j})\right|\leq
C\left( {\delta^{-2j}\over t^2} \right)^k \wedge \left( {t\over \delta^{-j}} \right)
 {(t\vee \delta^{-j})^{\eta} \over \left(t\vee \delta^{-j}+ |y-y_{Q_{\tau}^j}| \right )^{n+\eta}}.
\end{eqnarray*}
Note also for some $\lambda<\min\{\eta, k\}$,
\begin{eqnarray*}
 \sum\limits_j\sum\limits_{\tau\in I_j}\left( {\delta^{-2j}\over t^2} \right)^k \wedge \left( {t\over \delta^{-j}} \right)
 {(t\vee \delta^{-j})^\eta \over (t\vee \delta^{-j}+ |y-y_{Q_{\tau}^j}|  )^{n+\eta}   }
 \left(1+ { |x-y_{Q_{\tau}^j}|\over  \delta^{-j} }  \right)^{\lambda} |Q_{\tau}^j|\leq C.
\end{eqnarray*}
 This  yields
\begin{eqnarray*}
&&\hspace{-0.6cm}\left|t\nabla e^{-t^2L} f(y)\right| \\
 &&\leq  C \sum\limits_j\sum\limits_{\tau\in I_j}\big|\big\langle T_\delta^{-1}f(\cdot),
 (\delta^{-2j}L)^k e^{-\delta^{-2j}L}(\cdot,y_{Q_{\tau}^j})\big\rangle\big|
\left( {\delta^{-2j}\over t^2} \right)^k \wedge \left( {t\over \delta^{-j}} \right)
 {(t\vee \delta^{-j})^\eta \over \left(t\vee \delta^{-j}+ |y-y_{Q_{\tau}^j}|  \right)^{n+\eta}   } |Q_{\tau}^j|\\
   &&\leq  C \sup_{j, \tau\in I_j}\left|\big\langle T_\delta^{-1}f(\cdot),
 (\delta^{-2j}L)^k e^{-\delta^{-2j}L}(\cdot,y_{Q_{\tau}^j})\big\rangle\right|
 \left(1+ { |x-y_{Q_{\tau}^j}|\over  \delta^{-j} }  \right)^{-\lambda}\\
  &&\hskip 3.1cm \times \sum\limits_j\sum\limits_{\tau\in I_j}\left( {\delta^{-2j}\over t^2} \right)^k \wedge
 \left( {t\over \delta^{-j}} \right)
 {(t\vee \delta^{-j})^\eta \over \left(t\vee \delta^{-j}+ |y-y_{Q_{\tau}^j}|  \right)^{n+\eta}   }
 \left(1+ { |x-y_{Q_{\tau}^j}|\over  \delta^{-j} }  \right)^{\lambda} |Q_{\tau}^j|\\
 &&\leq  C \sup_{j, \tau\in I_j}\left|\big\langle T_\delta^{-1}f(\cdot),
 (\delta^{-2j}L)^k e^{-\delta^{-2j}L}(\cdot,y_{Q_{\tau}^j})\big\rangle\right|
 \left(1+ { |x-y_{Q_{\tau}^j}|\over  \delta^{-j} }  \right)^{-\lambda}.
\end{eqnarray*}
Using Theorem 2.3 in \cite{CT},  we decompose the cubes $\{Q_{\tau}^j\}$
into annuli according to the distance $|x-y_{Q_{\tau}^j}|$ with respect to $\delta^{-2j}$
to obtain
\begin{align*}
\left\| \sup_{|x-y|< 2t} |t\nabla e^{-t^2L} f(y)|  \right\|_{L^1(\mathbb{R}^n)}
&\leq C
 \left\|  \sup_{j }\sup_{ \tau\in I_j}\big|\big\langle T_\delta^{-1}f(\cdot),
 (\delta^{-2j}L)^k e^{-\delta^{-2j}L}(\cdot,y_{Q_{\tau}^j})\big\rangle\big| \left(1+ { |x-y_{Q_{\tau}^j}|\over  \delta^{-j} }
   \right)^{-\lambda}   \right\|_{L^1(\mathbb{R}^n)}\\
&\leq C\|N_h\left(T_\delta^{-1}f\right) \|_{L^1(\mathbb{R}^n)}\\
&= C\|T_\delta^{-1}f\|_{H_L^1(\mathbb{R}^n)}\\
 &\leq C\|f\|_{H_L^1(\mathbb{R}^n)},
\end{align*}
since $\lambda>n$.  By a  density argument, we obtain \eqref{e lemma maximal}. Hence,
the  proof of  Lemma~\ref{le6.3} is complete.
\end{proof}

\bigskip

\noindent
{\bf Proof of Theorem~\ref{th5.1}}.\  By (i) of \eqref{e5.1}, it suffices to show
\begin{align}\label{radial maximal}
\|N_hf\|_1\leq C\|f^+_h\|_1.
\end{align}
To prove this, we use the ideas in \cite{CT}.
We claim that, there exists a constant $A$, independent of $f$ and $r$, such that the following inequality holds:
\begin{align}\label{e5.7}
\|N_hf\|_1\leq  2A r^{-n}\|f^+_h\|_1+ r\Big\|\sup\limits_{|x-y|<2t}|t\nabla e^{-t^2L}f(y)|\Big\|_1
\end{align}
for any $r\in (0,1]$.
If  the  claim \eqref{e5.7} is proven,   we then  apply Lemma \ref{le6.3} to get
$
\|N_hf \|_1\leq 2A r^{-n}\|f^+_h\|_1+ r C_2\|N_hf \|_1.
$
Finally, we can choose $r$ small enough so that $C_2 r<1$ to obtain
\begin{align*}
\|N_hf \|_1\leq \frac{2A}{1-rC_2} r^{-n}\|f^+_h\|_1
\end{align*}
as desired.

To prove the  claim \eqref{e5.7}, it suffices  to show  the following inequality:
\begin{align}\label{good lambda ineq}
\Big|\{N_hf>s\} \cap \{\sup\limits_{|x-y|<2t}|t\nabla e^{-t^2L}f(y)| \leq sr^{-1}\}\Big|
&\leq A r^{-n}\big|\{ f^+_h>s/2\}\big|,
\end{align}
for any $0<r\leq 1$ and $s>0$. Indeed, if (\ref{good lambda ineq}) is proven,  one   writes
\begin{align*}
\|N_hf \|_1&=\int_0^\infty \big| \{N_hf>s \} \big| \, ds \\
&\leq A r^{-n}\int_0^\infty  |\{ f^+_h>s/2\}| \, ds+ \int_0^\infty
 |\{\sup\limits_{|x-y|<2t}|t\nabla e^{-t^2L}f(y)|  > sr^{-1}\}| ds\nonumber\\
&= 2A r^{-n}\big\|f^+_h\big\|_1+r\big\|\sup\limits_{|x-y|<2t}|t\nabla e^{-t^2L}f(y)|\big\|_1.
\end{align*}

  Now,  let us  verify  (\ref{good lambda ineq}).
  Let $x_0\in \{N_hf>s\} \cap \{\sup\limits_{|x-y|<2t}|t\nabla e^{-t^2L}f(y)| \leq s/r\}$.
  Then  there exist  $y_0\in {\mathbb R}^n$ and $t_0>0$  such that $|x_0-y_0|<t_0$,
$|e^{-t_0^2L}f(y_0)|>s$ and  $|t_0\nabla e^{-t_0^2L}f(z)|\leq s/r$  whenever $|x_0-z|<2t_0$.
Note that
\begin{align*}
|e^{-t_0^2L}f(z)|&\geq |e^{-t_0^2L}f(y_0)|-|e^{-t_0^2L}f(y_0)-e^{-t_0^2L}f(z)|\\
&\geq |e^{-t_0^2L}f(y_0)|-|\nabla e^{-t_0^2L}f(\xi)||y_0-z| \nonumber
\end{align*}
for some $\xi$ which lies in between $y_0$ and $z$.
If $|z-y_0|<\frac{r}{2}t_0$ and $r\in (0, 1]$, then $|z-x_0|< 2t_0$. It tells us that
\begin{align*}
|e^{-t_0^2L}f(z)|\geq |e^{-t_0^2L}f(y_0)|-\frac{s}{t_0r}|z-y_0|\geq s- \frac{s}{2}=\frac{s}{2},
\end{align*}
which implies that $B(y_0, \frac{rt_0}{2})\subset  \{ f^+_h>s/2 \}$. This, in combination with the fact
that
$$
\frac{|B(x_0,2t_0)\cap \{ f^+_h>s/2 \}|}{|B(x_0,2t_0)|} \geq \frac{|B(y_0, \frac{t_0r}{2})|}{|B(x_0,2t_0)|}=\frac{r^n}{4^n},
$$
yields
\begin{align}\label{e6.6}
\{N_hf>s\} \cap \left\{ \sup\limits_{|x-y|<2t}|t\nabla e^{-t^2L}f(y)| \leq s/r\right \} \subset
\left\{\mathcal{M} \chi_{\{f^+_h>s/2\}} > {r^n}/{4^n} \right \}.
\end{align}
By the weak (1,1) boundedness of   $\mathcal{M}$,
$$
\left|\{N_hf>s\} \cap \left\{\sup\limits_{|x-y|<2t}|t\nabla e^{-t^2L}f(y)| \leq s/r \right \} \right |
\leq A r^{-n} \left |\{f^+_h>s/2\} \right |,
$$
which  proves  (\ref{good lambda ineq}).     The proof  of Theorem~\ref{th5.1} is complete.
   \hfill{}$\Box$

\bigskip

\noindent
{\bf Acknowledgments.} \
The authors would like to thank the referees for carefully reading the manuscript and for
offering  valuable suggestions, which made the paper complete, accurate and more readable.

\end{document}